\documentclass[12pt,a4paper]{article}
\usepackage{amsmath,amscd,amssymb,amsthm,amstext,mathrsfs,dcolumn,
booktabs,enumerate}

\usepackage{geometry}
\usepackage{ifpdf}
\ifpdf
\usepackage[all,line,cmtip]{xy}      
\else
\usepackage[all,line,cmtip,dvips]{xy}
\fi

\CompileMatrices

\usepackage[final]{showkeys}

\input{xypic}

\theoremstyle{plain}
   \newtheorem{theorem}{Theorem}[section]
   \newtheorem*{theorem*}{Theorem}
   \newtheorem{proposition}[theorem]{Proposition}
   \newtheorem{lemma}[theorem]{Lemma}
   \newtheorem{corollary}[theorem]{Corollary}
   \theoremstyle{definition}
   
   \newtheorem{definition}[theorem]{Definition}
   
   \theoremstyle{remark}
   
   \newtheorem{remark}[theorem]{Remark}

\newcommand{\ZZ}{{\mathbb Z}}
\newcommand{\FF}{{\mathbb F}}

\newcommand{\Cpos}{C^{+}}
\newcommand{\Cneg}{C^{-}}
\newcommand{\Cper}{C^{per}}
\newcommand{\D}{{\partial^-}}
\newcommand{\odelta}{{\overline \delta}}
\newcommand{\oq}{{\overline q}}
\newcommand{\ophi}{{\overline \phi}}
\newcommand{\ou}{{\overline u}}
\newcommand{\tensor}{{\otimes}}
\newcommand{\multneg}{{\mu^-}}

\newcommand{\ellper}{{\ell^{per}}}
\newcommand{\ellpos}{{\ell^+}}

\newcommand{\drca}{{\Omega_{A|k}^*}}
\newcommand{\ella}{{\ell (A)}}
\newcommand{\ellposa}{{\ell^+ (A)}}
\newcommand{\ellpera}{{\ell^{per} (A)}}
\newcommand{\Lfa}{{{\mathcal L}(A)}}
\newcommand{\Bfa}{{{\mathcal B}(A)}}
\newcommand{\vnulla}{{V_0(A)}}

\DeclareMathOperator{\image}{Im}
\DeclareMathOperator{\Ker}{Ker}

\DeclareMathOperator{\sgn}{sgn}
\DeclareMathOperator{\Hom}{Hom}
\DeclareMathOperator{\Map}{\textbf{Map}}

\newcommand{\scat}{{\Delta}}
\newcommand{\Alg}{{\textbf{Alg}}}
\newcommand{\Commalg}{{\textbf{Commalg}}}
\newcommand{\Mod}{{\textbf{Mod}}}
\newcommand{\Set}{{\textbf{Set}}}
\newcommand{\olb}{{\overline B}}

\newcommand{\Commalgnu}{{\textbf{Commalg}_e}}
\newcommand{\ellnu}{{\ell_e}}
\newcommand{\ellpernu}{{\ell_e^{per}}}

\title{A Hochschild-Kostant-Rosenberg theorem for cyclic homology}
\author{Marcel B\"okstedt \& Iver Ottosen}

\begin{document}
\maketitle

\begin{center}
Department of Mathematical Sciences, Aarhus University, \\
Ny Munkegade 118, 8000 Aarhus C, Denmark. \\
marcel@math.au.dk
\end{center}

\begin{center}
Department of Mathematical Sciences, Aalborg University, \\
A.C. Meyers V\ae nge 15, 2450 K\o benhavn SV, Denmark. \\
ottosen@math.aau.dk
\end{center}

\begin{abstract}
Let $A$ be a commutative algebra over the field $\FF_2= \ZZ / 2$. 
We show that there is a natural algebra homomorphism $\ell (A) \to HC^-_*(A)$ which is an isomorphism when 
$A$ is a smooth algebra. Thus, the functor $\ell$ can be viewed as an approximation of negative cyclic homology 
and ordinary cyclic homology $HC_*(A)$ is a natural $\ell (A)$-module. In general, there is a spectral sequence 
$E^2 = L_*(\ell )(A) \Rightarrow HC_*^- (A)$. We find associated approximation functors 
$\ellpos$ and $\ellper$ for ordinary cyclic homology and periodic cyclic homology, and
set up their spectral sequences. Finally, we discuss universality of the approximations.

\begin{flushleft}MSC: 19D55; 18G50\end{flushleft}
\end{abstract}

\section{Introduction}
\label{sec:Intro}

In this paper we continue the study of the $\ell$-functor introduced in \cite{BO}. This is a functor 
from the category of $\FF_2$-algebras with extra structure to the category of $\FF_2$-algebras, which is defined by
generators and relations. The functor was invented to give information about the cohomology of the free loop space $LX$. It comes with a natural transformation
\[ 
\ell(H^*(X; \FF_2)) \to H^*_{S^1}(LX; \FF_2) := H^*(ES^1\times_{S^1} LX; \FF_2)
\]
and it is further related to the equivariant cohomology of the free loop space by a spectral 
sequence of the form
\[ E_2^{-m,t}= H_m(H^*(X; \FF / 2) ; \ell )^t \Rightarrow H^*_{S^1}(LX; \FF_2) \]  
when $X$ is simply connected (\cite{BO_Topology} Theorem 7.4). 
Here $H_m(-; \ell )$ denotes the $m$th non-abelian left derived functor of $\ell$. 
In some cases one can prove that this spectral sequence collapses. An additional advantage of this functor is that it can easily 
be lifted to an endofunctor on the category of unstable algebras over the Steenrod algebra.

There is also a version of the $\ell$-functor at odd primes, but we feel that our understanding of the odd case is inferior to
our understanding of the functor in characteristic 2.

There is a similar but different functor that has been studied with the same purpose in mind, namely negative cyclic homology.
In an abstract sense, this functor completely describes the cohomology of the free loop space, since there is a
chain homotopy equivalence between the group of cochains $C^*(ES^1\times_{S^1} LX)$ and the negative cyclic homology of the ring of cochains $C^*(X)$. The drawback of this isomorphism  is that the ring of cochains is often hard to work with, especially in positive characteristic. But in the special situation that  the cochain complex of the space $X$ is formal, negative cyclic homology gives the complete results.

As far as we are aware, there is no spectral sequence relating negative cyclic homology of $H^*(X,R)$ to the ring 
$H^*_{S^1}(LX;R)$. Also, there does not appear to exist a natural map 
$HC_*^-(H^*(X;\FF_2)) \to H^*_{S^1}(LX, \FF_2)$. 
Finally, it is difficult to incorporate the action of the Steenrod algebra into the definition of cyclic homology over $\FF_2$.

To sum up, the situation is that in characteristic 2 we have two distinct functors, each  with  some strong and some weak points, which both aim at describing the equivariant cohomology of $LX$. The purpose of this paper is to relate the two functors to each other, in the hope that they can do a better job working together than either of them manages to do by themselves.

The cohomology of the free loop space is on our minds at all times. However, in this paper we discuss the purely algebraic relation between negative cyclic homology and the $\ell$-functor. The point of departure is the Hochshild-Kostant-Rosenberg theorem, which states that for a smooth algebra $A$, there is an explicit isomorphism between the Hochschild homology $HH_*(A)$ and the algebraic de Rham complex on $A$. Our point of view is that all flavors of cyclic homology are closely related to Hochschild homology, and the $\ell$-functor is equally closely related to the algebraic de Rham complex. Without making a precise statement, we note the analogy to the relation of the equivariant cohomology $H^*_{S^1}(LX)$ to the non-equivariant cohomology $H^*(LX)$.

Our main result is that for a unital, augmented, smooth algebra $A$ satisfying appropriate finiteness conditions, there is an isomorphism of rings $\ell(A)\to HC^-_* (A)$. This is a result about the product structures, because it follows from well known results on $HC^-_* (A)$ that there is a vector space isomorphism.
In order to prove this result, we construct a natural transformation
$\ell(A) \to HC^-_*(A)$. We think of this map as a purely  algebraic analogue
of the map $\ell(H^*(X; \FF_2)) \to H^*_{S^1}(LX; \FF_2)$.

There is a way of defining negative cyclic homology for simplicial rings \cite{Goodwillie}. It is known that $HC^- (A)$ is a homotopy invariant functor of $A$ (\cite{Goodwillie} I.3.5. with $\FF_2$-coefficients instead of $\ZZ$). That is, if we replace $A$ by a homotopy equivalent simplicial ring $A_\bullet$, the negative cyclic homologies of $A$ and $A_\bullet$ will agree. Furthermore, it is known that for any $\FF_2$-algebra, we can find a homotopy equivalent simplicial algebra $A_\bullet$ such that $A_n$ is a polynomial algebra for every $n$. By filtering the chain complex for $HC^-_*(A_\bullet)$ after the simplicial
direction, one gets a spectral sequence with $E^1=HC^-_*(A_n)$, which is similar to Quillen's fundamental spectral sequence converging towards Hochschild homology. This spectral sequence has a multiplicative structure. Since polynomial algebras are smooth, our main result applies to give $E^1\cong \ell(A_n)$, as a ring, and we see that this is the same $E_1$-term as in the spectral sequence converging towards $H^*_{S^1}(LX)$. Therefore, we think of the fundamental spectral sequence as its algebraic analogue. We do not know the relation between these two spectral sequences.  

To sum up the situation as we understand it presently, $\ell$ is related to the cohomolology of free loop spaces and to negative cyclic homology in very similar fashions. In both cases we have a natural map, and a spectral sequence. The relation between $HC_*^- (H^*(X))$ and $H^*_{S^1}(LX)$ seems less clear. However, if $H^*(X)$ is a polynomial algebra, both natural maps are isomorphisms (for the free loop space case, see \cite{BO}).

The definition of negative cyclic homology works with equal ease in all characteristics, and even relative to a base ring. In contrast, at the moment we have only defined $\ell$ for algebras over $\FF_2$. We intend to use the connection to cyclic homology to work out the correct generalization of $\ell$ in future work. This is particularly important in odd characteristic.

\section{Cyclic homology theories}
\label{sec:CHT}
We briefly introduce Hochschild-, cyclic-, negative cyclic- and periodic homology of algebras in this section together with
their product structures in the commutative case. References are \cite{Goodwillie}, \cite{Loday} and \cite{NEH}.

Let $k$ be a commutative ring and let $A$ be an associative and unital $k$-algebra. Put $\overline A = A/k$ 
and define 
\[ C_n(A)=
\begin{cases}
A\otimes {\overline A}^{\otimes n} , & n\geq 0, \\
0, & n<0.
\end{cases} \]
Write $(a_0, a_1, \dots , a_n)$ for the element $a_0\otimes a_1\otimes \dots \otimes a_n$.
The Hochschild boundary map $b:C_n(A) \to C_{n-1}(A)$ is defined by
\[ b(a_0, \dots , a_n) = \sum_{i=0}^{n-1} (-1)^i (a_0, \dots , a_i\cdot a_{i+1}, \dots ,a_n) + 
(-1)^n (a_n\cdot a_0, \dots , a_{n-1}). \]
It satisfies $b^2 =0$ such that we have a chain complex $C_* (A)$. The Hochschild homology of $A$ is the
homology of this complex $HH_n (A) = H_n(C_*(A))$.
 
There is also a Connes' boundary map $B:C_n(A) \to C_{n+1}(A)$ defined by
\[ B(a_0, \dots , a_n) = \sum_{i=0}^n (-1)^{ni} (1, a_i , \dots , a_n, a_0, \dots , a_{i-1}). \]
One has $B^2=0$ and $bB+Bb=0$.
 Let $\olb_{**}(A)$ denote the $\ZZ \times \ZZ$-graded bicomplex with
\[ \olb_{p,q} (A) = C_{q-p} (A) \]
and boundary maps given by 
\[ B: \olb_{p,q} (A) \to \olb_{p-1,q} (A) , \quad b: \olb_{p,q} (A) \to \olb_{p,q-1} (A). \]
This bicomplex is periodic. It is a module over the polynomial algebra $k[u]$ where $u$ has bidegree $(-1,-1)$. The 
generator acts by sending $x\in \olb_{p,q} (A) = C_{q-p}(A)$ to $ux=x\in \olb_{p-1,q-1}(A) = C_{q-p}(A)$.

For $-\infty \leq \alpha \leq \beta \leq +\infty$ we let $T^{\alpha , \beta}_* (A)$ denote the $\ZZ$-graded 
chain complex with
\[  T^{\alpha , \beta }_n (A) = \prod_{\alpha \leq p \leq \beta } \olb_{p,n-p} (A) \]
and boundary map $B+b$. Note that we use the direct product here. If $-\infty < \alpha$ only finitely many 
non-zero factors appear such that the direct product can be replaced by a direct sum.

\begin{definition}
\label{def:homologies}
Hochschild-, cyclic-, negative cyclic and periodic homology of $A$ are defined as follows:
\begin{center}
\begin{tabular}{l l}
$HH_*(A)= H_*(T^{0,0}_*(A)), \quad $ & $HC_*(A)=H_*(T^{0, \infty}_*(A)),$ \\
$HC^-_* (A)=H_*(T^{-\infty ,0}_*(A)),\quad $ & $HC^{per}_*(A)=H_*(T^{-\infty ,\infty }_*(A)).$
\end{tabular}
\end{center}
\end{definition}

It is sometimes convenient to use a different notation as follows:
\begin{align*}
& \Cpos_* (A) = k[u,u^{-1}]/uk[u] \otimes C_*(A) \cong T_*^{0,\infty}(A), \\
& \Cneg_* (A) = k[u] \widehat \otimes C_*(A) \cong T_*^{-\infty , 0}(A), \\
& \Cper_* (A) = k[u,u^{-1}]\widehat \otimes C_*(A) \cong T_*^{-\infty, \infty}(A),
\end{align*}
Here $u$ and $u^{-1}$ have lower degrees $-2$ and $2$ respectively and $\widehat \otimes$ denotes
the completion of the graded tensor product $\otimes$. The isomorphisms are given by
\[ u^i \otimes c_{n-2i} \mapsto c_{n-2i}\in \olb_{i,n-i}(A), \]
and the differentials on the chain complexes $\partial^+$, $\partial^-$, $\partial^{per}$ by
\[ id \otimes b+ u\cdot (id\otimes B). \]

If $A$ is a {\em commutative} $k$-algebra there are product structures which we now describe. 
The product on Hochschild homology comes from the shuffle map (\cite{Loday} Section 4.2).
The symmetric group $S(n)$ acts from the left on $C_n (A)$ by
\[ \sigma \cdot (a_0, a_1, \dots , a_n) = 
(a_0 , a_{\sigma^{-1} (1)} , a_{\sigma^{-1}(2)} ,\dots , a_{\sigma^{-1} (n)} ) . \]
Let $S(p,q)\subseteq S(p+q)$ denote the set of $(p,q)$-shuffles.
The shuffle map 
\[ sh: C_*(A) \otimes C_*(A) \to C_* (A\otimes A) \]
is a chain map defined by
\begin{align*}
sh((a_0, a_1, \dots , a_p) & \otimes (b_0, b_1, \dots , b_q) ) = \\
& \sum_{\tau \in S(p,q)} \sgn (\tau ) \tau \cdot 
(a_0\otimes b_0, a_1\otimes 1 , \dots , a_p\otimes 1, 1\otimes b_1 , \dots ,1\otimes b_q ). 
\end{align*}
For a commutative algebras $A$ with multiplication $m_A$ the composite map
\[
\xymatrix@C=1.5 cm{
C_* (A) \otimes C_*(A)  \ar[r]^-{sh} & C_* (A\otimes A) \ar[r]^-{C_*(m_A)} & C_* (A)
}\] 
induces the product on Hochschild homology $HH_*(A)\otimes HH_*(A) \to HH_*(A)$.

The product on negative cyclic homology comes from the shuffle map and a cyclic shuffle map (\cite{Loday} Section 4.3).
A cyclic $(p,q)$-shuffle is a permutation $\sigma \in S(p+q)$ obtained as follows:
First perform a cyclic permutation of any order on the set $\{1 , \dots , p \}$ and a cyclic permutation of any 
order on the set $\{ p+1, \dots , p+q\}$. Then shuffle the two results to obtain $\{ \sigma (1), \dots , \sigma (p+q) \}$.
If $1$ appears before $p+1$, then $\sigma$ is a cyclic $(p,q)$-shuffle. Let $CS(p,q)$ denote the set of cyclic $(p,q)$-shuffles.

The cyclic shuffle map
\[ sh^\prime : C_* (A) \otimes C_* (A) \to C_{*+2} (A\otimes A) \]
is given by
\begin{align*}
sh^\prime ((a_0, a_1, & \dots , a_p ) \otimes (b_0, b_1, \dots , b_q )) = \\
& \sum_{\sigma \in CS(p+1,q+1)} \sgn (\sigma ) \sigma^{-1} \cdot (1\otimes 1, a_0\otimes 1, \dots , 
a_p\otimes 1, 1\otimes b_0, \dots , 1\otimes b_q ).
\end{align*}
Note that one has to use $\sigma^{-1}$ here as opposed to $\sigma$ which appears in \cite{Loday}. See 
\cite{KR} page 45 regarding this.

By combining the shuffle and cyclic shuffle maps one gets a chain map as follows:
\[ Sh = id\otimes sh+u\cdot (id\otimes sh^\prime ) : 
k[u]\widehat \otimes C_* (A)\otimes C_*(A) \to k[u] \widehat \otimes C_* (A\otimes A). \]
When $A$ is commutative, we can define $\multneg : \Cneg_*(A) \otimes \Cneg_*(A) \to \Cneg_*(A)$ as the following composite \cite{NEH} 1.12:
\[
\xymatrix@C=1.2 cm{
k[u]\widehat \otimes C_* (A)\otimes k[u]\widehat \otimes C_*(A) \ar[r]^-{id\otimes tw \otimes id} &  
k[u]\widehat \otimes k[u] \widehat \otimes C_*(A) \otimes C_*(A) \ar[r]^-{m_{k[u]}\otimes id}  & \\
k[u]\widehat \otimes C_*(A)\otimes C_*(A) \ar[r]^-{Sh} & 
k[u]\widehat \otimes C_*(A\otimes A) \ar[r]^-{id\otimes C_*(m_A)}  &
k[u]\widehat \otimes C_*(A)
}\] 
This chain map induces the product on negative cyclic homology
\[ HC^-_*(A) \otimes HC^-_*(A) \to HC^-_*(A). \]

For periodic cyclic and cyclic homology one can write up similar composite maps $\mu^{per}$ and $\mu^+$
which induce a product and a module structure as follows
\cite{Loday} 5.1.13: 
\begin{align*} 
& HC_*^{per}(A)\otimes HC_*^{per}(A) \to HC_*^{per}(A), \\
& HC_*^{-}(A)\otimes HC_*(A) \to HC_*(A).
\end{align*}

\begin{remark}
\label{product_formula}
By the description above we have the following multiplication formula which works in all three cases:
\begin{align*}
u^i\otimes (a_0, a_1, \dots ,a_p) \otimes u^j\otimes & (b_0, b_1, \dots ,b_q) \mapsto \\ 
u^{i+j} & \otimes \sum_{\tau \in S(p,q)} \sgn (\tau ) \tau \cdot (a_0b_0, a_1, \dots , a_p, b_1, \dots , b_q) \\
+u^{i+j+1} & \otimes \sum_{\sigma \in CS(p+1,q+1)} \sgn (\sigma ) \sigma^{-1} \cdot 
(1, a_0, \dots, a_p, b_0, \dots , b_q).
\end{align*}
Note that all the terms of the second summation are zero if $a_0=1$ or $b_0=1$.
\end{remark}

The short exact sequence $0\to T_*^{-\infty , -1} \to T_*^{-\infty , 0} \to T_*^{0,0} \to 0$ of chain complexes
gives us the long exact sequence for negative cyclic homology.
\[
\xymatrix@C=1.0 cm{
\dots \ar[r] & HC^-_{*+2}(A) \ar[r]^-{\cdot u} & HC^-_* (A) \ar[r]^-{h} & HH_* (A) \ar[r]^-{\partial} 
& HC^-_{*+1} (A) \ar[r] & \dots  
}\] 
Here the map $h$ is an algebra homomorphism. 

Likewise the short exact sequence $0\to T_*^{0,0} \to T_*^{0,\infty } \to T_*^{1,\infty } \to 0$ gives us 
Connes' long exact sequence
\[
\xymatrix@C=1.0 cm{
\dots \ar[r] & HH_* (A) \ar[r]^-{I} & HC_* (A) \ar[r]^-{\cdot u} & HC_{*-2} (A) \ar[r]^-{\partial} 
& HH_{*-1} (A) \ar[r] & \dots  
}\] 
The short exact sequence $0\to T_*^{-\infty , 0} \to T_*^{-\infty , \infty } \to T_*^{1, \infty } \to 0$ gives us 
the long exact sequence for periodic cyclic homology.
\[
\xymatrix@C=1.0 cm{
\dots \ar[r] & HC^-_* (A) \ar[r]^-{\iota} & HC^{per}_* (A) \ar[r]^-{S} & HC_{*-2} (A) \ar[r]^-{\partial} 
& HC^-_{*-1} (A) \ar[r] & \dots  
}\] 
Here $\iota$ is an algebra homomorphism.

We view $HH_*(A)$ and $HC_*^{per}(A)$ as $HC_*^-(A)$-modules via the algebra homomorphisms $h$ and $\iota$.
One has the following result:

\begin{proposition}
\label{prop:HC-mod}
The three long exact sequences above are long exact sequences of $HC_*^- (A)$-modules.
\end{proposition}

\begin{proof}
The proposition is a consequence of a more general result which we first describe. Let $R_*$ be a chain complex
of $k$-modules equipped with at chain map $\mu : R_*\otimes R_* \to R_*$ which induces a product in homology
\[ H_*(R) \otimes H_*(R) \to H_*(R\otimes R) \to H_*(R) \]
such that $H_*(R)$ becomes a graded $k$-algebra.
Here the map to the left is the canonical map given by $[c]\otimes [c^\prime ] \mapsto [c\otimes c^\prime ]$
and the map to the right is induced by $\mu$. Let
\[
\xymatrix@C=1.0 cm{
0 \ar[r] & L_* \ar[r]^-{i} & M_* \ar[r]^-{p} & N_* \ar[r] & 0  
}\]
be a short exact sequence of chain complexes of $k$-modules. 

Assume that we have chain maps
$\alpha : R_*\otimes L_* \to L_*$, $\beta: R_*\otimes M_* \to M_*$ and $\gamma : R_*\otimes N_* \to N_*$
which induces $H_*(R)$-module structures on $H_*(L)$, $H_*(M)$ and $H_*(N)$ when 
the canonical maps are composed by the induced maps of these chain maps. Assume furthermore, that
we have commutative diagrams as follows:
\[
\xymatrix@C=1.0 cm{
R_* \otimes L_*  \ar[r]^-{\alpha} \ar[d]^-{id\otimes i} & L_* \ar[d]^- {i} 
& R_* \otimes M_* \ar[r]^-{\beta} \ar[d]^-{id\otimes p} & M_* \ar[d]^-{p} \\
R_* \otimes M_* \ar[r]^-{\beta} & M_* 
&  R_* \otimes N_* \ar[r]^-{\gamma} & N_*
}\]

Then the long exact sequence of homology groups
\[
\xymatrix@C=1.0 cm{
\dots \ar[r] & H_*(L) \ar[r]^-{i_*} & H_*(M) \ar[r]^-{p_*} & H_*(N) \ar[r]^-{\partial_*} & H_{*-1}(L) \ar[r] & \dots  
}\] 
is a long exact sequence of $H_*(R)$-modules. 

By applying the homology functor to the two commutative diagrams one gets that $i_*$ and $p_*$ are $H_*(R)$-linear. 
So one only has to verify that the connecting homomorphism $\partial_*$ is $H_*(R)$-linear:
\[ \partial_*( r \cdot n) = (-1)^{|r|} r\cdot \partial_* (n), \quad r \in H_*(R), \quad n \in H_*(N). \]
This follows however by the diagram chase description of $\partial_*$ given in the snake lemma.

Let $R_* = C_*^- (A)$ with $\mu = \mu^-$ (note that $C_*^- (A)$ has the structure of a standard 
$A_\infty$-algebra by \cite{Ge-J} Theorem 4.4). 

Using the alternative notation the first short exact sequence, which we consider, takes the form
\[
\xymatrix@C=1.0 cm{
0 \ar[r] & C_{*+2}^- (A)  \ar[r]^-{\cdot u} & C_*^- (A) \ar[r]^-{h} & C_*(A) \ar[r] & 0  
}\] 
where  $h(u^i \otimes \underline a) = 1\otimes \underline a$ for $i=0$ and $h(u^i \otimes \underline a) = 0$ 
for $i>0$ and $\underline a \in C_*(A)$. Let $L_* = C_{*-2}^- (A)$ and
$M_* = C_*^- (A)$ with chain maps $\alpha$ and $\beta$ given by $\mu^-$. Put $N_* = C_*(A)$ and 
let $\gamma$ be the composite
\[
\xymatrix@C=1.0 cm{
C(A)_*^- \otimes C_*(A)  \ar[r]^-{id \otimes h} & C_*(A)\otimes C_*(A) \ar[r] & C_*(A) 
} \] 
where the chain map to the right is the shuffle map. Finally, let $i=\cdot u$ and $p=h$. 
By Remark \ref{product_formula} one sees that there are commutative diagrams as required and the result follows.

The second short exact sequence takes the form
\[
\xymatrix@C=1.0 cm{
0 \ar[r] & C_*(A)  \ar[r]^-{I} & C_*^+ (A) \ar[r]^-{\cdot u} & C_{*-2}^+ (A) \ar[r] & 0  
}\] 
where $I$ is the natural subcomplex inclusion. In this case $\beta$ and $\gamma$ both equals the
chain map $\mu^+$. Remark \ref{product_formula} gives us that the required diagrams commute 
such that we have the desired result.

The third short exact sequence takes the form
\[
\xymatrix@C=1.0 cm{
0 \ar[r] & C_*^- (A)  \ar[r]^-{\iota} & C_*^{per} (A) \ar[r]^-{S} & C_{*-2}^+ (A) \ar[r] & 0  
}\] 
where $\iota$ is the natural inclusion and $S(u^i \otimes \underline a) = u^{i+1} \otimes \underline a$ for $i\leq -1$
and $S(u^i \otimes \underline a) = 0$ for $i\geq 0$. In this case $\beta = \mu^{per} \circ (\iota \otimes id)$. Remark \ref{product_formula}
gives us that the required diagrams commute and we have the desired result.
\end{proof}

Hochschild and cyclic homology are also defined for graded $k$-algebras. One has to introduce various signs in the
description above, but that does not concern us, since we are only interested in the case $k=\FF_2=\ZZ /(2)$.  
We use the standard convention regarding {\em upper-} and {\em lower gradings}. If $M=M^*$ is a 
$\ZZ$-graded $k$-module in upper degrees its lower grading is given by $M_{-n} = M^n$, 
$n\in \ZZ$ and vice versa.

Let $A=A^*$ be a non-negatively graded commutative $\FF_2$-algebra with unit 
$\eta : k \to A^0$. The (upper) degree of a homogeneous element $a\in A$ is denoted $|a|$.
Put $\overline A = A/\eta (k)$ and write $s\overline A$ for $\overline A$ with the degree shift given by 
$|sa| = |a|-1$ Let $C_* (A)$ denote the normalized Hochschild complex with
\[ C_n (A) = A \otimes (s\overline A)^{\otimes n}.\]
In the graded case we use the bar complex notation and write $a_0[a_1| \dots |a_n]$ 
for the element $a_0\otimes sa_1\otimes \dots \otimes sa_n$. 
We refer to $n$ as the {\em homological degree} of this element. 
As a particular case, $a_0[]$ denotes $a_0 \in A = C_0 (A)$. The descriptions of differentials 
and product structures given above remain valid.
An ungraded $k$-algebra is viewed as a graded $k$-algebra concentrated ind degree zero. For such an 
algebra, the lower degree of an element in $C_n(A)$ equals the homological degree $n$.

\section{Filtrations and spectral sequences}
\label{spectral_sequence}

Let $k$ be a commutative ring and let $A$ be an associative and unital $k$-algebra. 
Put $T_* = T_*^{\alpha , \beta}(A)$. This is a chain complex with a $k[u]$-action. 
We have a filtration by subcomplexes as follows:
\[ \dots \subseteq F_{s}T_* \subseteq F_{s+1}T_* \subseteq \dots \subseteq T_* \]
where
\[ F_sT_n = \{ x \in T_n | \medspace x_{i,n-i} = 0 \text{ for } i \geq s+1 \} .\]
Consider the associated spectral sequence $\{ E^r_{**}(A, \alpha, \beta ) \}$ of $k[u]$-modules. We have
\[ E_{s,t}^1 (A, \alpha , \beta )= H_{s+t}(F_sT_*/ F_{s-1}T_*) = \begin{cases}
HH_{t-s}(A), & \alpha \leq s\leq \beta, \\
0, & \text{otherwise}.
\end{cases} \]
When $\alpha +1\leq s \leq \beta$, the $d^1$-differential is induced by Connes' boundary
\[ d^1: E_{s,t}^1 \to E_{s-1,t}^1; \quad d^1 = B_* : HH_{t-s}(A) \to HH_{t-s+1}(A). \]
The $k[u]$-action is given by $u\cdot = id : E^1_{s,t} \to E^1_{s-1, t-1}$.

We need this spectral sequence for negative cyclic homology.
In this case the filtration is not bounded below so we must consider convergence issues. 
Here our reference is \cite{Boardman}.

\begin{proposition} 
\label{prop_spectral_sequence}
We have a conditionally convergent spectral sequence
\[ E_{*,*}^r(A, -\infty, 0) \Rightarrow HC^-_* (A). \]
If for each $s$ and $t$ only finitely many of the differentials $d^r: E_{s,t}^r \to E_{s-r, t+r-1}^r$ are 
non-zero, then the spectral sequence converges strongly. If we have strong convergence, then the composite map
\[ HC_t^- (A) \twoheadrightarrow E_{0,t}^\infty \hookrightarrow E_{0,t}^1 = HH_t(A) \]
equals the canonical map $h: HC_t^-(A) \to HH_t(A)$ for all $t$.
\end{proposition}

\begin{proof}
We shift to the notation used in \cite{Boardman} by defining $F^iT_*=F_{-i}T_*$. So we have a filtration 
$\dots \subseteq F^2T_* \subseteq F^1T_* \subseteq F^0T_* = T_*$ which exhaust $T_*$. 
By \cite{Boardman} Theorem 9.3 page 29, the associated spectral sequence converges conditionally 
to the homology of the completion $\widehat T_*$. This chain complex is defined as the inverse limit 
of the sequence of projection maps
\[ \dots \to T_*/F^2T_* \to T_*/F^1T_* \to T_*/F^0T_* = 0. \]
But for every $s\geq 0$ and every $n$ we have a commutative diagram with vertical isomorphisms as follows:
\[
\xymatrix@C=1.5 cm{
T_n/F^{s+1}T_n \ar[r]^-{pr} \ar[d]_-{\cong} & T_n/F^sT_n \ar[d]^-{\cong} \\ 
{\prod_{-s \leq i \leq 0} \olb_{i,n-i} (A)} \ar[r]^-{pr} & {\prod_{-s+1 \leq i \leq 0} \olb_{i,n-i} (A) }
} \] 
So $\widehat T_n = \lim_s (T_n/F^sT_n) \cong \lim_s \prod_{-s<i\leq 0} \olb_{i,n-i} (A) = T_n$.
The differential on $\widehat T_*$ corresponds to the differential on $T_*$ under this isomorphism.

By the remark following \cite{Boardman} Theorem 7.1 page 20, conditional convergence and the extra assumption 
on the differentials ensure strong convergence. 

Finally, by the short exact sequence $0 \to F^1T_* \to T_* \to T_*/F^1T_* \to 0$ of chain complexes we have a long
exact sequence in homology, which shows that the natural map $h:HC_*^-(A) \to HH_*(A)$ factors 
through $H_*(T_*)/(\image (H_*(F^1T_*) \to H_*(T_*))$. In the case of strong convergence, 
this quotient is isomorphic to $E_{0,*}^\infty$.
\end{proof}

\begin{proposition}
\label{prop:smooth_alg_E2}
If $A$ is a smooth and commutative $k$-algebra then 
\[
E_{s,t}^2(A, -\infty , 0) \cong \begin{cases}
H_{DR}^{t-s}(A), & s<0, \\
\Ker (d: \Omega_{A|k}^t \to \Omega_{A|k}^{t+1}), & s=0, \\
0, & s>0.
\end{cases}
\]
\end{proposition}

\begin{proof}
We have an antisymmetrization map 
$\epsilon_* : \Omega_{A|k}^* \to HH_*(A)$ which is an algebra homomorphism since $A$ is commutative
\cite{Loday} 1.3.12 page 27 and 4.2.9 page 126. Furthermore by \cite{Loday} 2.3.3 page 69, there is a 
commutative diagram 
\[
\xymatrix@C=1.5 cm{
\Omega_{A|k}^n \ar[r]^-{d} \ar[d]_-{\epsilon_n} & \Omega_{A|k}^{n+1} \ar[d]^-{\epsilon_{n+1}} \\ 
HH_n(A) \ar[r]^-{B_*} & HH_{n+1}(A)
} \] 

Now if $A$ is smooth over $k$. Then the antisymmetrization map is an isomorphism by 
the HKR-theorem. Thus the differential $d^1 : E_{s,t}^1 \to E_{s-1,t}^1$ is given by the de Rham differential
$d: \Omega_{A|k}^{t-s} \to \Omega_{A|k}^{t-s+1}$. The result follows by taking homology.
\end{proof}

\section{The $\ell$-functor}
\label{sec:l-functor}

In this section we let $k = \FF_2$. In \cite{BO} Definition 5.1, the $\ell$-functor is defined for
graded commutative $k$-algebras equipped with a linear operator $\lambda$ which is a derivation over 
the Frobenius homomorphism in the sense that 
\[\lambda (ab)=\lambda (a) b^2+ a^2\lambda (b). \] 
When comparing to negative cyclic homology, we take $\lambda = 0$. In this case, the definition of 
the functor is as follows:

\begin{definition}
\label{def:ell}
Let $A$ be a commutative $k$-algebra. Then $\ell (A)$ is the free commutative and unital $k$-algebra 
on generators $\delta (a)$, $\phi (a)$, $q(a)$ for $a\in A$ and a single generator $u$ modulo the following 
relations for $a, b, c\in A$:
\begin{align}
& \phi (a+b) = \phi (a) + \phi (b), \label{r1} \\
& \delta (a+b) = \delta (a) + \delta (b), \label{r2} \\
& q(a+b) = q(a)+q(b)+\delta (ab), \label{r3} \\
& \delta (ab)\delta (c)+ \delta (bc)\delta (a)+ \delta(ca) \delta (b) = 0, \label{r4} \\
& \phi (ab) = \phi (a)\phi (b) + uq(a)q(b), \label{r5} \\
& q(ab) = q(a)\phi (b)+\phi (a)q(b), \label{r6} \\
& \delta (a)^2 = 0, \label{r7} \\
& q(a)^2 = 0, \label{r8} \\
& \delta (a) \phi (b) = \delta (ab^2), \label{r9} \\
& \delta (a) q(b) = \delta (ab) \delta (b), \label{r10} \\
& u\delta (a) = 0, \label{r11} \\
& \phi (1) = 1. \label{r12}
\end{align}
If $A=A^*$ is a graded commutative $k$-algebra then so is $\ell (A)$ where the (upper) grading is given by
$|\delta (a)| = |a|-1$, $|\phi (a)| = 2|a|$, $|q(a)|=2|a|-1$ and $|u|=2$. The homological grading
is given by $||\delta (a)|| = 1$, $||\phi (a)|| = 0$, $||q(a)|| = 1$ and $||u||=-2$.
\end{definition}

\begin{remark}
\label{remark:ell}
We have $\delta (0) = \phi (0) = q(0) = 0$ by (\ref{r1}), (\ref{r2}) and (\ref{r3}). Furthermore, 
$\delta (a^2)=q(a^2)=0$ by (\ref{r3}) and (\ref{r6}) such that $\delta (1) = q(1)=0$. 
In \cite{BO_Topology} Definition 7.2 the functor $\ell$, for $\lambda =0$, appears slightly different as a quotient of a 
free commutative non-unital $k$-algebra where the relation $\phi (1) u = u$ is used instead of (\ref{r12}).
The two definitions are however equivalent as one sees by inserting $b=1$ in the relations (\ref{r5}), (\ref{r6}) and (\ref{r9}).
If $A$ is non-negatively graded and connected (that is $A^0 = k$) then $\ell (A)$ is non-negatively graded. 
\end{remark}

We now let $A$ be a commutative and unital $k$-algebra. There is a chain complex
\begin{equation}
\label{ell_cc}
\xymatrix@C=1.0 cm{
\dots \ar[r] & \ell (A) \ar[r]^-{\cdot u} & \ell (A) \ar[r]^-{r} & \Omega_{A|k} \ar[r]^-{\tau} & \ell (A) \ar[r] & \dots  
} 
\end{equation}
which models the long exact sequence for negative cyclic homology
\[
\xymatrix@C=1.0 cm{
\dots \ar[r] & HC^-_{*+2} (A) \ar[r]^-{\cdot u} & HC^-_* (A) \ar[r]^-{h} & HH_*(A) \ar[r]^-{\partial} & 
HC^-_{*+1}(A) \ar[r] & \dots  
} \] 
We will now define the maps in this chain complex and describe their properties. By \cite{BO} Theorem 8.2 we have:

\begin{proposition}
\label{proposition_r}
There is a natural algebra homomorphism
\[ r: \ell (A) \to \Omega^*_{A|k}; \quad \delta (a) \mapsto da, \quad q(a) \mapsto ada, \quad 
\phi (a) \mapsto a^2, \quad u\mapsto 0. \]
(One can remember what it does on generators by the three names de Rham, Cartier and Frobenius.)
Note that $d\circ r=0$. 
\end{proposition}

The following result also appears in \cite{BO}, but we have reformulated it slightly.

\begin{proposition}
There is a natural $k$-linear transfer map 
\[ \tau : \Omega^*_{A|k} \to \ell (A); \quad a_0 da_1 \dots da_n \mapsto \delta (a_0) \delta (a_1) \dots \delta (a_n),
\quad a_0 \mapsto \delta (a_0) .\]
It satisfies $r\circ \tau = d : \Omega_{A|k}^* \to \Omega_{A|k}^*$ and $\tau \circ r = 0$. Frobenius reciprocity holds
\[ \tau ( r(\alpha ) \beta ) = \alpha \tau (\beta ) .\]
Finally, the image of $\tau$ is the following ideal of $\ell (A)$:
\[ I_\delta (A) = (\delta (a) | a\in A). \]
\end{proposition}

\begin{proof}
We first show that $\tau$ is well-defined. By the property $\Omega^*_{A|k} = \Lambda^*_A (\Omega^1_{A|k} )$ 
and relation (\ref{r7}) it suffices to show that $\tau :\Omega^1_{A|k} \to \ell (A)$ is well defined.
But $\tau$ respects the relations 
\begin{align*} 
& ad(b+c)=adb+adc, \\
& ad(bc) = abd(c)+ acd(b) 
\end{align*}
since we have the relations (\ref{r2}) and (\ref{r4}). So $\tau$ is well-defined. 
We have that $r\circ \tau = d$ because $r(\delta (a))= da$. 

It suffices to verify Frobenius reciprocity in the special cases where $\alpha$ is a generator of $\ell (A)$ and 
$\beta = a_0da_1\dots da_n$. For $\alpha = \phi (a)$ we find
\begin{align*}
\tau (r(\phi (a))a_0 da_1\dots da_n) & =
\tau(a^2a_0da_1\dots da_n) = \delta (a^2a_0) \delta(a_1) \dots \delta (a_n) \\
& = \phi (a) \delta (a_0) \delta (a_1) \dots \delta (a_n) = \phi (a) \tau (a_0 da_1\dots da_n)
\end{align*} 
by relation (\ref{r9}). Similarly relation (\ref{r10}) shows that Frobenius reciprocity holds when $\alpha = q(a)$.
For $\alpha = \delta (a)$ it follows directly by the definition of $\tau$. Finally for $\alpha = u$ the left hand side is zero since
$r(u)=0$ and the right hand is zero by relation (\ref{r11}).

If we put $\beta = 1$ in the Frobenius reciprocity equation we find that $\tau \circ r=0$ since $\tau (1) = \delta (1) = 0$.
By relation (\ref{r9}), (\ref{r10}) and (\ref{r11}) the image of $\tau$ is the ideal $(\delta (a) | a\in A)$ in $\ell (A)$.
\end{proof}

Note that $r\circ u=0$, $\tau \circ r=0$ and $u\circ \tau=0$ such that we do have a chain complex (\ref{ell_cc}) 
as mentioned above. We will now describe a quotient of $\ell (A)$ in terms of the de Rham complex on $A$ and
use this to examine how much exactness we have in the chain complex.

\begin{definition}
\label{tilde_ell}
Define the following functor
\[ \tilde \ell (A) = \ell (A)/I_\delta (A) .\] 
\end{definition}

We can describe $\tilde \ell (A)$ as the free commutative and unital $k$-algebra on generators $\phi (a)$, $q(a)$ 
for $a\in A$ and $u$ modulo the relations
\begin{align*}
&\phi (a+b) = \phi (a)+\phi (b), &  & q(a+b) = q(a)+q(b), \\
&\phi (ab) = \phi (a) \phi (b) + uq(a)q(b),& & q(ab)=q(a)\phi (b)+\phi (a)q(b), \\
&q(a)^2 = 0, & & \phi (1) = 1.
\end{align*}

\begin{remark}
There is a surjective algebra homomorphism
\[ f: \tilde \ell (A) \to \Omega_{A|k}^* ; \quad \phi (a) \mapsto a, \quad q(a)\mapsto da, \quad u \mapsto 0 \]
and the induced map 
\[ \tilde f : \tilde \ell (A)/u\tilde \ell (A) \to \Omega_{A|k}^* \]
is an isomorphism by the description above.
\end{remark}

The algebra homomorphism $f$ has a $k$-linear section
\[ s: \Omega_{A|k}^* \to \tilde \ell (A); \quad s(a_0 da_1\dots da_n) = \phi (a_0) q(a_1) \dots q(a_n), \quad 
s(a_0) = \phi (a_0) \]
such that $f\circ s = id$.
 
\begin{proposition}
Define a bilinear map
\[ \{ \cdot , \cdot \} : \Omega_{A|k}^* \otimes \Omega_{A|k}^* \to \Omega_{A|k}^*; \quad
\{ a , b\} = da\cdot db. \]
Then $(\Omega_{A|k}^* , \{ \cdot , \cdot \} )$ is a Poisson algebra. That is $\{ \cdot , \cdot \}$ is a Lie bracket 
which is a derivation in the sense that
\[ \{ a, b\cdot c \} = \{ a, b \} \cdot c + b \cdot \{ a , c \}  \]
for all $a, b, c \in \Omega_{A|k}^*$.  Extend $\{ \cdot , \cdot \}$ to a bracket
on the polynomial algebra $\Omega_{A|k}^* [u]$ by
\[ \{ua, b\} = \{ a, ub \} = u \{ a, b \} .\]
Define a deformation quantization $*$ of the product $\cdot$ on $\Omega_{A|k}^* [u]$ as follows:
\[ a*b=a\cdot b +u \{ a , b \} . \]
The product $*$ satisfies the associative, commutative and distributive laws. Furthermore, one has
\[ (x_0dx_1\dots dx_n)*(y_0dy_1\dots dy_m) = (x_0*y_0)dx_1\dots dx_ndy_1\dots dy_m. \]
\end{proposition}

\begin{proof}
We have that $\{ a, a\} = da\cdot da =0$ and the Jacobi identity holds since
$\{ a, \{ b , c \} \} = da \cdot d(db\cdot dc)=0$. Furthermore,
\[ \{ a, b\cdot c \} = da\cdot d(b\cdot c) = da\cdot db \cdot c+ da\cdot b\cdot dc = \{  a, b \} \cdot c + b\cdot \{ a , c \} . \]
The product $*$ is commutative and distributive over addition. It is also associative since
\begin{align*}
a*(b*c) & = a*(bc+u\{ b, c\} ) = abc+u\{ a, bc \} + ua*\{ b, c \} \\
& = abc+u(\{ a, b\} c+b\{ a, c\} + a\{ b, c\} ) 
\end{align*}
which equals
\[ (a*b)*c = (ab+u\{ a, b\} )*c = abc + u\{ ab, c\} + u \{ a, b\} *c. \]
The last formula follows directly from the definition of the product $*$ and the bracket
\end{proof}

\begin{theorem}
\label{theorem_deformation}
There is an isomorphism of $k$-algebras
\[
\xymatrix@C=0.5 cm{
\overline f: \tilde \ell (A) \ar[r]^-{\cong} & (\Omega_{A|k}^* [u], *); & \phi (x) \mapsto x, 
& q(x) \mapsto dx, & u \mapsto u  
} \] 
\end{theorem}

\begin{proof}
We first verify that $\overline f$ is a well-defined algebra homomorphism. The transformation rules for 
$\overline f$ define an algebra homomorphism $\hat f$ from the free commutative and unital algebra on generators 
$\phi (a)$, $q(a)$ for $a\in A$ and $u$. 
We must verify that $\hat f$ preserves the relations for $\tilde \ell (A)$ described after Definition \ref{tilde_ell}.

The first two additive relations and the last relation are preserved. So are the remaining three relations by the 
following computations:
\begin{align*}
\hat f (\phi (ab)) &= a\cdot b = a*b+u\{ a, b\} =a*b+u(da*db) \\
&= \hat f (\phi (a) \phi(b) +uq(a)q(b)), \\
\hat f (q(ab)) &= d(ab) = da\cdot b+a\cdot db = da*b+a*db \\
& = \hat f (q(a)\phi (b)+\phi (a)q(b)), \\
\hat f (q(a)^2) &= da*da = da\cdot da+u\{ da, da \} = 0. 
\end{align*}
Next we construct an inverse map $\overline s: (\Omega_{A|k}^* [u], *) \to \tilde \ell (A)$. On $\Omega_{A|k}^*$
we define $\overline s$ to agree with the $k$-linear section $s$. We extend this to all of $\Omega_{A|k}^*[u]$
by the formula $\overline s (ux)=u\overline s(x)$. Note that $\overline s(1) = s(1) = \phi (1) = 1$. 
The following computation shows that $\overline s$ is multiplicative

\begin{align*}
& \overline s ((x_0dx_1\dots dx_n)*(y_0dy_1\dots dy_m)) = 
\overline s ((x_0*y_0)dx_1\dots dx_ndy_1\dots dy_m) \\
= & \overline s ((x_0y_0+udx_0dy_0)dx_1\dots dx_ndy_1\dots dy_m) \\
= & \phi (x_0y_0)q(x_1)\dots q(x_n)q(y_1)\dots q(y_m)+uq(x_0)\dots q(x_n)q(y_0)\dots q(y_m) \\
= &  \phi (x_0)\phi (y_0)q(x_1)\dots q(x_n)q(y_1)\dots q(y_m)\\
= & \overline s (x_0dx_1\dots dx_n)\overline s (y_0dy_1\dots dy_m ).
\end{align*}
We have $\overline f \circ \overline s = id$ but also $\overline s \circ \overline f = id$ as one verifies on
the generators $\phi (a)$, $q(a)$ and $u$.
\end{proof}

\begin{corollary}
In the chain complex (\ref{ell_cc}) one has that $\Ker (\cdot u)=\image (\tau )$.
\end{corollary}

\begin{proof}
We must show that the map $\cdot u: \ell (A)/\image (\tau )\to \ell (A)$ is injective. 
Its domain equals $\tilde \ell (A)$ and the composite 
\[ \xymatrix@C=0.5 cm{\tilde \ell (A) \ar[r]^-{\cdot u} & \ell (A) \ar[r] & \tilde \ell (A) } \]
is simply multiplication by $u$ on $\tilde \ell (A)$. By the theorem above it suffices to show that
multiplication by $u$ on $(\Omega_{A|k}^*[u], *)$ is injective. This algebra is $\Omega_{A|k}^*[u]$ 
with a new multiplicative structure. But the multiplication by $u$ has not changed in this new structure and
it is injective in the polynomial algebra $\Omega_{A|k}^* [u]$.
\end{proof}

\begin{definition}
\label{ell_mod_u}
Let $u\ella \subseteq \ella$ be the ideal generated by $u$. Define the functor
\[ \Lfa = \ell (A)/ u\ell (A) \]
and let $\Bfa \subseteq \Lfa$ denote the ideal
\[ \Bfa = ( \delta (a) | \medspace a \in A). \]
\end{definition}

\begin{proposition}
Consider the filtration $\ell (A) \supseteq u\ell (A) \supseteq u^2\ell (A) \supseteq \dots $. Its associated graded object is
\[ Gr_*(\ell (A)) \cong {\cal L} (A) \oplus (\Omega_{A|k}^* [u])^{>0} \]
where the last summand denotes the ideal of $\Omega_{A|k}^* [u]$ generated by $u$.
\end{proposition}

\begin{proof}
By definition, $Gr_0(\ell (A)) = \ell (A)/ u\ell (A) = {\cal L} (A)$. For $i>0$ the canonical projection
$\ell (A) \to \tilde \ell (A)$ restricts to a surjection $u^i \ell (A) \to u^i\tilde \ell (A)$. But by relation (\ref{r11}) 
we have trivial intersection $I_\delta (A) \cap u^i\ell (A) = \{ 0  \}$. So $u^i\ell (A) \to u^i\tilde \ell (A)$
is an isomorphism. By the 5-lemma we get an associated isomorphism
\[ Gr_i (\ell (A)) = \frac {u^i \ell (A)} {u^{i+1} \ell (A)} \to 
\frac {u^i \tilde \ell (A)} {u^{i+1} \tilde \ell (A)} = Gr_i (\tilde \ell (A)). \]
Finally, the theorem above gives us that $Gr_i (\tilde \ell (A)) \cong u^i \Omega_{A|k}^*$.
\end{proof}

Recall that the Cartier map is the following algebra homomorphism:
\[\Phi : \Omega_{A|k}^* \to H_{DR}(A); \quad a \mapsto a^2, \quad da\mapsto ada . \]  
When the Cartier map is an isomorphism we have further de Rham complex interpretations which we now describe.

The map $r:\ell (A) \to \Omega_{A|k}^*$ satisfies $d\circ r=0$ and $r(u)=0$. So it induces an algebra
homomorphism $\overline r : {\cal L} (A) \to \Ker (d)$. Note also that $\overline r (\Bfa ) \subseteq \image (d)$. 

\begin{theorem}
\label{theorem_ker(d)}
If the Cartier map $\Phi : \Omega_{A|k}^* \to H_{DR}(A)$ is an isomorphism, then the homomorphism
$\overline r: \Lfa \to \Ker (d)$ and its restriction $\Bfa \to \image (d)$ are also isomorphisms.
\end{theorem}

\begin{proof}
Assume that $\Phi$ is an isomorphism. Then $\overline r$ is an isomorphism by \cite{BO} Theorem 8.5.
Put $\widetilde \Omega (A) = \Lfa / \Bfa$. Its description by generators and relations shows that we have an 
isomorphism $\widetilde \Omega (A) \cong \drca$; $\phi (x)\mapsto x$, $q(x)\mapsto dx$.There is 
a commutative diagram, where the rows are short exact sequences
\[ \xymatrix@C=1.0 cm{
0 \ar[r] & \Bfa \ar[r] \ar[d] & \Lfa \ar[r] \ar[d]^-{\overline r} & {\widetilde \Omega} (A) \ar[r] \ar[d] & 0 \\
0 \ar[r] & \image (d) \ar[r] & \Ker (d) \ar[r] & H_{DR}(A) \ar[r] & 0.   
} \]
The right vertical map corresponds to the Cartier map and is thus also an isomorphism. 
The result follows by a diagram chase.
\end{proof}

In consequence, we find
\begin{proposition}
\label{graded_object_special}
If the Cartier map $\Phi: \Omega_{A|k}^* \to H_{DR}(A)$ is an isomorphism, then we have a long exact 
sequence 
\[ \xymatrix@C=1.0 cm{
\dots \ar[r] & \ell (A) \ar[r]^-{\cdot u} & \ell (A) \ar[r]^-{r} & \Omega_{A|k} \ar[r]^-{\tau} & \ell (A) \ar[r] & \dots  
} \]
Furthermore, the filtration 
$\ell (A) \supseteq u\ell (A) \supseteq u^2\ell (A) \supseteq \dots $ has the associated graded object
\[ Gr_*(\ell (A)) \cong \Ker (d) \oplus (\Omega_{A|k}^* [u])^{>0} .\]
\end{proposition}

\begin{proof}
We have already seen, that the sequence is a chain complex in general. Furthermore one has:

$\image (\cdot u) = \Ker (r)$: The cokernel of $\cdot u$ is the algebra ${\cal L} (A)$. So we must show that
the map ${\cal L} (A)  \to \Omega_{A|k}^*$ is injective. This follows by the theorem above.

$\image (r) = \Ker (\tau )$: Let $x\in \Ker (\tau )$. Since $r\circ \tau =d$ we see that $dx=0$. By the theorem above
this implies that $x$ is in the image of $r$. 

$\image (\tau ) = \Ker (\cdot u)$: This true in general by the corollary above.

The description of the graded object follows directly by the proposition and theorem above.
\end{proof}

We conclude this section with a lemma regarding the the length of the filtration.

\begin{lemma}
\label{finite_filtration}
Let $A=A^*$ be a non-negatively graded commutative $\FF_2$-algebra such that $A^0$ is finite. Then the filtration
\[ \ell (A) \supseteq u\ell (A) \supseteq u^2 \ell (A) \supseteq \dots \]
is finite in each degree. That is, for each $n$, the filtration
\[ [ \ell (A)]^n \supseteq [u\ell (A)]^n \supseteq [u^2 \ell (A)]^n \supseteq \dots \]
is finite. In consequence, $\lim_i (u^i \ell (A)) = 0$ and $\lim_i^1 (u^i\ell (A)) = 0$.
\end{lemma}

\begin{proof}
For $i>0$, the $\FF_2$-vector space $u^i \ell (A)$ is generated by the elements of the form
\[ g = u^j \phi (a_1) \dots \phi (a_r) q(b_1)\dots q(b_s) \]
where $j \geq i$ and $b_h\neq b_k$ for $h\neq k$. This follows by the relations $u\delta (a)=0$ and $q(a)^2=0$.
The degree of such an element is
\begin{align*} 
|g| & = 2j+2|a_1|+\dots +2|a_r| + 2|b_1|-1 +\dots +2|b_s|-1 \\
& \geq 2j-\# \{k| b_k \in A^0 \} \\
& \geq 2i-\# (A^0). 
\end{align*}
Thus, 
\[ [u^i \ell (A)]^n =0 \text{ for } i > \frac {n+\# (A^0)} 2 . \]
For each $n$ we have 
\[ [\lim_i (u^i \ell (A))]^n = \lim_i [u^i \ell (A)]^n = 0 .\]
So the inverse limit is trivial. Since $[u^i \ell (A)]^n \supseteq [u^{i+1} \ell (A)]^n$ becomes the surjection $0=0$ from 
the stage given above, we have 
\[ [{\lim_i}^1 (u^i \ell (A)) ]^n = {\lim_i}^1 [u^i \ell (A)]^n =0 \]
by \cite{Boardman} Proposition 1.8. So the ${\lim}^1$ term is also trivial.
\end{proof}

\section{The approximation theorem}
\label{sec:themap}

\begin{definition}
We say that a unital $k$-algebra $A$ is {\em supplemented} if it is equipped with an augmentation $A\to k$ such that 
the composite $k\to A\to k$ is the identity.
\end{definition}

\begin{theorem}
\label{thm:lfunctorapprox}
Let $A$ be a commutative and unital (possibly non-negatively graded) algebra over $k=\FF_2$. 
Then there is a natural algebra homomorphism as follows:
\begin{align*} 
\psi: \ell (A) \to HC_*^- (A); \quad & \delta (a) \mapsto 1\otimes 1[a], \\ 
& q(a) \mapsto 1\otimes a[a], \\ 
& \phi (a) \mapsto 1\otimes a^2[] +u \otimes 1[a|a], \\
& u\mapsto u\otimes 1[].
\end{align*}
Furthermore, one has a commutative diagram
\[
\xymatrix@C=1.0 cm{
\dots \ar[r] & \ell (A) \ar[r]^-{\cdot u} \ar[d]^-{\psi} & \ell (A) \ar[d]^-{\psi} \ar[r]^-{r} 
& \Omega_{A|k}^* \ar[r]^-{\tau} \ar[d]^-{\epsilon} & \ell (A) \ar[r] \ar[d]^-{\psi} & \dots  \\
\dots \ar[r] & HC^-_{*+2} (A) \ar[r]^-{\cdot u} & HC^-_* (A) \ar[r]^-{h} & HH_*(A) \ar[r]^-{\partial} & 
HC^-_{*+1}(A) \ar[r] & \dots  
}\] 
where the upper row is a chain complex with $\Ker (\cdot u) = \image (\tau )$ and the 
lower row is the long exact sequence for negative cyclic homology. 

If $A$ is smooth over $k$, supplemented and of finite type, then $\psi$ is an algebra isomorphism. 
In this case both rows of the diagram are exact and the vertical maps are isomorphisms.
\end{theorem}

\begin{proof}
We first show that there is a natural algebra homomorphism $\psi$ as stated. Name the image elements as follows:
\[ \odelta (a) = 1\otimes 1[a], \quad \oq (a) = 1\otimes a[a], \quad \ophi (a) = 1\otimes a^2 []+u\otimes 1[a|a], 
\quad \ou = u\otimes 1[] \] 
We must verify that they are cycles in the chain complex $\Cneg_*(A)$.

Since $b(1[a])=a[]+a[]=0$ and $B(1[a])=1[1|a]+1[a|1]=0$ we have that $\D (\odelta (a))=0$.

Since $b(a[a])=a^2[]+a^2[]=0$ and $B(a[a])=1[a|a]+1[a|a]=0$ we have that $\D (\oq (a))=0$.

We have $b(a^2[])=0$ and $B(a^2[])=1[a^2]$ such that $\D (a^2[])=u\otimes 1[a^2]$. Furthermore, 
$b(1[a|a])=a[a]+1[a^2]+a[a]=1[a^2]$ and $B(1[a|a])=0$ such that $\D (u\otimes 1[a|a]) = u\otimes 1[a^2]$.
Thus, $\D (\ophi (a))=0$.

Finally, $b(1[])=0$ and $B(1[])=0$ such that $\D (\ou )=0$.

Next we show that the relations (\ref{r1})-(\ref{r12}) are mapped to valid relations among the 
classes represented by $\odelta (a)$, $\oq (a)$, $\ophi (a)$ and $\ou$. 

(\ref{r1}): We have
\begin{align*}
\ophi (a+b) &= 1\tensor (a+b)^2[]+u\tensor 1[a+b|a+b] \\
&= 1\tensor a^2[]+1\tensor b^2[] + u \tensor (1[a|a]+1[b|b]+1[a|b]+1[b|a]) \\
&= \ophi (a) + \ophi (b) + u\tensor (1[a|b]+1[b|a]).
\end{align*}
The last term is a boundary since $b(a[b])=ab[]+ab[]=0$ and $B(a[b])=1[a|b]+1[b|a]$ such that
$\D (1\tensor a[b]) = u\tensor (1[a|b]+1[b|a])$.

(\ref{r2}): Is OK since
\[ \odelta (a+b) = 1\tensor 1[a+b]=1\tensor 1[a]+1\tensor 1[b]=\odelta (a)+\odelta (b). \]

(\ref{r3}): We have
\begin{align*}
\oq (a+b) &= 1\tensor (a+b)[a+b]=1\tensor a[a]+1\tensor b[b]+ 1\tensor a[b]+1\tensor b[a] \\
&= \oq (a) + \oq (b) + 1\tensor (a[b]+b[a]).
\end{align*}
Hence we must show that $\odelta (ab) +1\tensor (a[b]+b[a])=1\tensor (1[ab]+a[b]+b[a])$ is a boundary. But 
$b(1[a|b])=a[b]+1[ab]+b[a]$ and $B(1[a|b])=0$ so the element above is the boundary $\D (1\tensor 1[a|b])$.

(\ref{r4}): By the formula for the product $\multneg :\Cneg_1 (A)\otimes \Cneg_1 (A) \to \Cneg_2 (A)$ 
given in Remark \ref{product_formula} we find
\[ \odelta (ab) \odelta (c) = \multneg (1\otimes 1[ab]\otimes 1\otimes 1[c]) =1\otimes 1[ab|c]+1\otimes 1[c|ab]. \]
So we have
\begin{align*}
& \odelta (ab)\odelta (c) + \odelta (bc) \odelta (a) + \odelta (ca) \odelta (b) = \\
& 1\otimes 1[ab|c] + 1\otimes 1[c|ab] + 1\otimes 1[bc|a] + 1\otimes 1[a|bc] +1\otimes 1[ca|b] + 1\otimes 1[b|ca]. 
\end{align*}
But this expression equals $\D (1\otimes 1[a|b|c]+1\otimes 1[b|c|a] +1\otimes 1[c|a|b])$ by a direct computation.

(\ref{r5}): Firstly, we have
\begin{align*}
\ophi (a) \ophi (b) 
=\multneg \big( ( & 1\otimes a^2 [] +u\otimes 1[a|a])\otimes (1\otimes b^2[] + u\otimes 1[b|b]) \big) \\
=\multneg \big( & 1\otimes a^2[]\otimes 1\otimes b^2[]+1\otimes a^2[]\otimes u\otimes 1[b|b]\\
+&  u\otimes 1[a|a] \otimes 1\otimes b^2[]+u\otimes 1[a|a]\otimes u\otimes 1[b|b]
\big) .
\end{align*}
By Remark \ref{product_formula} we find
\begin{equation}
\label{small_product_formula}
 \multneg (u^i\otimes a_0[] \otimes u^j\otimes b_0[])=
u^{i+j}\otimes a_0b_0[]+u^{i+j+1}\otimes 1[a_0|b_0] . 
\end{equation}
Thus
\[ \multneg (1\otimes a^2[] \otimes 1\otimes b^2[])=1\otimes a^2b^2[]+u\otimes 1[a^2|b^2]. \]
By Remark \ref{product_formula} we also find
\begin{align*}
& \multneg (1\otimes a^2 []\otimes u\otimes 1[b|b]) = u\otimes a^2[b|b], \\
& \multneg (u\otimes 1[a|a]\otimes 1\otimes b^2[]) = u\otimes b^2[b|b],
\end{align*}
and
\begin{align*}
& \multneg (u\otimes 1[a|a]\otimes u\otimes 1[b|b]) = 
u^2 \otimes \sum_{\tau \in S(2,2)} \tau \cdot 1[a|a|b|b] = \\
& u^2\otimes (1[a|a|b|b]+1[a|b|a|b]+1[a|b|b|a]+1[b|a|a|b]+1[b|a|b|a]+1[b|b|a|a]).
\end{align*}
So we have
\begin{align*}
& \ophi (a) \ophi (b) = 1\otimes a^2b^2[] + u\otimes (1[a^2|b^2]+a^2[b|b]+b^2[a|a])\\
& +u^2\otimes (1[a|a|b|b]+1[a|b|a|b]+1[a|b|b|a]+1[b|a|a|b]+1[b|a|b|a]+1[b|b|a|a]).
\end{align*}
Secondly,
\begin{align*}
u\oq (a) \oq(b) &= \multneg (u\otimes a[a]\otimes 1\otimes b[b]) \\
&= u\otimes (ab[a|b]+ab[b|a])+u^2\otimes \sum_{\sigma \in CS(2,2)} \sigma^{-1} \cdot 1[a|a|b|b].
\end{align*}
The cyclic $(2,2)$-shuffles are the 12 permutations mapping $(1,2,3,4)$ to
\begin{align*}
& (1,2,3,4), (2,1,3,4), (1,2,4,3), (2,1,4,3), (1,3,2,4), (1,3,4,2), \\
& (1,4,2,3), (4,1,2,3), (1,4,3,2), (4,1,3,2), (2,4,1,3), (4,2,1,3). 
\end{align*}
Using this list, one finds
\[ \sum_{\sigma \in CS(2,2)} \sigma^{-1} \cdot 1[a|a|b|b] = 1[a|b|a|b]+1[b|a|b|a] \]
such that
\[ u\oq (a)\oq (b) = u\otimes (ab[a|b]+ab[b|a])+u^2\otimes (1[a|b|a|b]+1[b|a|b|a]). \]
Finally, $\ophi (ab) = 1\otimes a^2b^2[]+u\otimes 1[ab|ab]$ so we have
\begin{align*}
\ophi (ab) & +\ophi (a) \ophi (b) + u\oq (a) \oq (b) = \\
u & \otimes (1[ab|ab]+1[a^2|b^2]+ab[a|b]+ab[b|a]+a^2[b|b]+b^2[a|a]) \\
+u^2 & \otimes (1[a|a|b|b]+1[a|b|b|a]+1[b|b|a|a]+1[b|a|a|b]).
\end{align*}
We must show that this element is a boundary. A direct computation shows that it equals
\[ \D (u\otimes (1[a|b|ab]+1[a|a|b^2]+a[b|a|b]+a[a|b|b])) \]
so the relation is OK.

(\ref{r6}):  We have
\begin{align*}
\oq (a) \ophi (b) & = \multneg (1\otimes a[a] \otimes (1\otimes b^2[]+u\otimes 1[b|b])) \\
& = \multneg (1\otimes a[a] \otimes 1\otimes b^2[])+\multneg (1\otimes a[a] \otimes u\otimes 1[b|b])) \\
&= 1 \otimes ab^2[a] + u\otimes \sum_{\sigma \in CS(2,1)} \sigma^{-1} \cdot 1[a|a|b^2]
+u\otimes \sum_{\tau \in S(1,2)} \tau \cdot a[a|b|b].
\end{align*}
The cyclic $(2,1)$-shuffles are the three permutations mapping $(1,2,3)$ to 
$(1,2,3)$, $(1,3,2)$ and $(2,1,3)$. The $(1,2)$-shuffles are the three permutations mapping $(1,2,3)$ to 
$(1,2,3)$, $(2,1,3)$ and $(3,1,2)$. So
\begin{align*}
\oq (a) \ophi (b) &=  1 \otimes ab^2[a]+u\otimes (1[a|a|b^2]+1[a|b^2|a]+1[a|a|b^2])\\
&+u\otimes (a[a|b|b]+a[b|a|b]+a[b|b|a]) \\
&= 1 \otimes ab^2[a]+u\otimes (1[a|b^2|a]+a[a|b|b]+a[b|a|b]+a[b|b|a]).
\end{align*}
Thus we have
\begin{align*}
\oq (ab) + \oq (a) \ophi(b) + \ophi (a) \oq(b) &= 1\otimes ab[ab] + 1\otimes ab^2[a] + 1\otimes a^2b[b] \\
+u\otimes ( & 1[a|b^2|a]+a[a|b|b]+a[b|a|b]+a[b|b|a]\\
+ & 1[b|a^2|b]+b[b|a|a]+b[a|b|a]+b[a|a|b]).
\end{align*}
A direct computation shows that this element equals the boundary
\[ \D ( 1\otimes ab[a|b]+ u\otimes (1[a|b|b|a]+1[b|a|a|b]+1[b|a|b|a])) . \]
So the relation is OK.

(\ref{r7}): This relation is OK since
\[\odelta (a)^2 = \multneg (1\otimes 1[a]\otimes 1\otimes 1[a]) = 1\otimes 1[a|a]+1\otimes 1[a|a]=0.\]

(\ref{r8}): We have
\begin{align*}
\oq (a)^2 &= \multneg (1\otimes a[a]\otimes 1\otimes a[a]) \\
& = 1\otimes a^2[a|a]+1\otimes a^2[a|a] + u\otimes \sum_{\sigma \in CS(2,2)} \sigma^{-1} \cdot 1[a|a|a|a] 
\end{align*}
which is zero since there are 12 cyclic $(2,2)$-shuffles.

(\ref{r9}): We have
\begin{align*}
\odelta (a) \ophi (b) &= \multneg \big( 1\otimes 1[a]\otimes (1\otimes b^2[]+u\otimes 1[b|b]) \big) \\
&= 1\otimes \sum_{\tau \in S(1,0)} \tau \cdot b^2[a] + 
u \otimes \sum_{\tau \in S(1,2)} \tau \cdot 1[a|b|b] \\
&= 1\otimes b^2[a]+u\otimes (1[a|b|b]+1[b|a|b]+1[b|b|a])
\end{align*}
and $\odelta (ab^2)=1\otimes 1[ab^2]$. By a direct computation, one finds that
\[ \D (1\otimes 1[a|b^2]+1\otimes a[b|b]) = \odelta (a) \ophi (b) + \odelta (ab^2) . \]
So the relation is OK.

(\ref{r10}): We have
\begin{align*} 
& \odelta (a) \oq (b) = \multneg (1\otimes 1[a]\otimes 1\otimes b[b]) = 1\otimes (b[a|b]+b[b|a]), \\
& \odelta (ab) \odelta (b) = \multneg (1\otimes 1[ab]\otimes 1\otimes 1[b])=1\otimes (1[ab|b]+1[b|ab]). 
\end{align*}
Such that
\[ \D (1\otimes 1[b|a|b]) = \odelta (a) \oq (b)+ \odelta (ab) \odelta (b). \]

(\ref{r11}): This relation is OK since
\[ \ou \odelta (a) = u\otimes 1[a] = \D (1\otimes a[]) . \]

(\ref{r12}): We have
\[ \overline \phi (1) = 1\otimes 1[] + u\otimes 1[1|1]=1\otimes 1[] \]
which is the unit in $HC^-_*(A)$. 

Regarding the commutativity of the diagram, we recall that the antisymmetrization map is given by
\[ \epsilon : \Omega_{A|k}^* \to HH_*(A); \quad a_0da_1\dots da_n \mapsto
\sum_{\sigma \in S(n)} \sigma \cdot a_0[a_1|\dots |a_n]. \] 
Especially, $\epsilon (a_0) = a_0[]$ and $\epsilon (a_0da_1) = a_0[a_1]$. The middle square commutes since

\begin{align*}
& h\circ \psi (\delta (a))=h(1\otimes 1[a])=1[a] =\varepsilon (da) = \varepsilon \circ r (\delta (a)), \\
& h\circ \psi (q(a))=h(1\otimes a[a])=a[a]=\varepsilon (ada) = \varepsilon \circ r (q(a)), \\
& h\circ \psi (\phi (a))=h(1\otimes a^2[]+u\otimes 1[a|a])=a^2[]=\varepsilon (a^2)=\varepsilon \circ r (\phi (a)), \\
& h\circ \psi (u)=h(u\otimes 1[])=0=\varepsilon (0) = \varepsilon \circ r (u). 
\end{align*}

The left square commutes since $\psi$ is an algebra homomorphism. Regarding the right square, we find a
formula for the connecting homomorphism $\partial$ by the snake lemma. Consider the following diagram
where $T_* = C_*^- (A)$ and $C_* = C_*(A)$:

\[ \xymatrix@C=1.0 cm{
0 \ar[r] & T_* \ar[r]^-{\cdot u} \ar[d]^-{\partial^-} & T_* \ar[r]^-{h} \ar[d]^-{\partial^-} & C_* \ar[r] \ar[d]^-{b} & 0 \\
0 \ar[r] & T_* \ar[r]^-{\cdot u} & T_* \ar[r]^-{h} & C_* \ar[r] & 0. 
} \]

Let $x\in C_*$ such that $b(x)=0$. Then $1\otimes x\in T_*$ is a lift of $x$ and 
$\partial^- (1\otimes x) = 1  \otimes b(x) +u\otimes B(x) = u\otimes B(x)$. A lift of this element is $1\otimes B(x)\in T_*$
so we have $\partial (a) = 1\otimes B(a)$.

By Proposition \ref{prop:HC-mod} the connecting homomorphism $\partial$ is $HC_*^-(A)$-linear so
\[ \partial (h(y)x) = y\partial(x). \]
One can also verify this directly as follows:
Let $y\in T_*$ such that $\partial^-(y)=0$. Then $y\cdot (1\otimes x) \in T_*$ is a lift 
of the cycle $h(y)x\in C_*$ and
\[ \partial^- (y\cdot (1\otimes x)) = \partial^- (y)\cdot (1\otimes x)+y\cdot \partial^- (1\otimes x) =
y\cdot (u \otimes B(x)) \]
This element lifts to $y\cdot (1\otimes B(x))= y\partial (x)$ as desired.

We can write $a_0da_1\dots da_n\in \Omega_{A|k}^*$ as $r(\delta (a_1) \dots \delta (a_n)) a_0$. Furthermore, 
the commutativity of the middle diagram together with the $HC_*^- (A)$-linearity gives us that
\begin{align*}
& \psi\circ \tau (r(a)b) = \psi (a\tau(b)) = \psi (a) \psi (\tau (b)), \\
& \partial \circ \epsilon (r(a)b) = \partial (\epsilon (r(a))\epsilon (b)) = \partial (h(\psi (a))\epsilon (b)) = 
\psi (a) \partial (\epsilon (b)).
\end{align*}
Thus it suffices to show that $\psi (\tau (a_0)) = \partial (\epsilon (a_0))$ for $a_0\in A$, which is seen
as follows: $\psi (\tau (a_0))=\psi (\delta (a_0)) = 1\otimes 1[a_0] = 1\otimes B(a_0[]) = \partial (a_0[]) = 
\partial (\epsilon (a_0)).$

The chain complex property and partial exactness of the upper row of the diagram was proven in the previous section.

Finally, under the additional assumptions on $A$, the antisymmetrization map $\epsilon$ is an isomorphism by the 
Hochschild-Kostant-Rosenberg theorem. Furthermore, the Cartier map 
$\Phi: \Omega_{A|k}^*\to  H_{DR}(A)$ is an isomorphism by a classical result of 
Cartier \cite{C}, \cite{Ka} Theorem 7.2. The finiteness condition on $A$ ensures that the spectral sequence from 
Proposition \ref{prop_spectral_sequence} is strongly convergent. Its $E^2$-page is described at the end of 
section \ref{spectral_sequence}. The description of the composite map given in Proposition \ref{prop_spectral_sequence} together with the commutativity of the middle square, which we just proved, and Theorem \ref{theorem_ker(d)}, shows 
that there are no higher differentials on the $0$th-column. So $E_{0,*}^2=E_{0,*}^\infty$. 

The unit and augmentation condition on $A$ implies that the class $u$ survives to the $E^\infty$-page. 
By the $k[u]$-module structure of the spectral sequence, all higher differentials vanish. So $E^2=E^\infty$ in general.

By the definition of the filtration $F_sT_*$ of $T_*=\Cneg_* (A)$ from section \ref{spectral_sequence} and 
the way the periodicity class $u$ acts, one sees that the algebra homomorphism $\psi$ respects the filtrations in the sense that
$\psi (u^i\ell (A)) \subseteq F_{-i}T_*$. The induced map of associated graded objects is an isomorphism since it agrees
with the isomorphism described in Proposition \ref{graded_object_special}. By Lemma \ref{finite_filtration} the 
filtration of $\ell (A)$ is exhaustive, Hausdorff and complete in the sense of \cite{Boardman}. So is the filtration of 
$HC_*^- (A)$ since the spectral sequence converges strongly. By \cite{Boardman} Theorem 2.6 we see that 
$\psi$ is an isomorphism. (By Proposition \ref{graded_object_special} we also get the exactness of the upper row in 
the diagram directly).
\end{proof}

\begin{remark}
When defining the homomorphism $\psi : \ell (A) \to HC_*^- (A)$, the image class of the generator $\phi (a)$ might 
not be immediate to guess. This image class is however a square in $\Cneg_* (A)$, since by 
(\ref{small_product_formula}) one has
\[ \multneg(1\otimes a[] \otimes 1\otimes a[])=
1\otimes a^2[]+u\otimes 1[a|a]. \]
\end{remark}

\begin{remark}
Basic examples of smooth algebras over $k=\FF_2$ are polynomial algebras. If $V$ is a non-negatively graded 
vector space of finite type, then the symmetric algebra $S(V)$ is smooth. Other examples are finite fields of order
$2^r$, $r>0$ which are smooth according to \cite{Loday} Example 3.4.3.
\end{remark}

\section{Cyclic homology and periodic cyclic homology}

In this section we describe approximation functors $\ellpos$ for ordinary cyclic homology and $\ellper$ for
periodic cyclic homology. Let $k=\FF_2$ and let $A=A^*$ be a non-negatively graded commutative and unital $k$-algebra.
In particular, $A$ may be concentrated in degree zero such that it is an ungraded $k$-algebra. We first consider 
cyclic homology. 

\begin{definition}
Let $\ellpos (A)$ be the free $\ell (A)$-module on generators $\gamma (a)$ for $a\in A$ and $v^i$ for $i= 0, 1, 2, \dots$
modulo the following relations:
\begin{align}
& \gamma (a+b) = \gamma (a)+\gamma (b), \label{rp1} \\
& \phi (a) \gamma (b) = \gamma (a^2 b), \label{rp2} \\
& q(a) \gamma (b) = \delta (a) \gamma (ab), \label{rp3} \\
& \delta (a)\gamma (b) = \gamma (a) \delta (b), \label{rp4} \\
& \gamma (a) \delta (bc)+\gamma (ab) \delta (c) + \gamma (ac) \delta (b) = 0, \label{rp5} \\
& u\gamma (a)=0, \label{rp6} \\
& \delta (a) v^i =0, \quad i\geq 0, \label{rp7} \\
& uv^i = v^{i-1}, \quad i\geq 1, \label{rp8} \\
& \gamma (1) = v^0 \label{rp9} 
\end{align}
The upper degrees of the generators are $|\gamma (a)|=|a|$, $|v^i|=-2i$. The homological degrees
are $||\gamma (a)||=0$, $||v^i||=2i$.
\end{definition}

\begin{remark}
\label{ellpos_functor}
$\ellpos$ is a functor. If $f:A\to B$ is an algebra homomorphism then $\ellpos (f): \ellpos (A) \to \ellpos (B)$ is the 
$k$-linear map defined by $\gamma (a) \mapsto \gamma (f(a))$, $v^i\mapsto v^i$ and the requirement that
$\ellpos (f)$ is $\ella$-linear when the action on $\ell (B)$ is via the algebra homomorphism 
$\ell (f) : \ell (A) \to \ell (B)$.
\end{remark}

\begin{remark}
The de Rham complex $\drca$ is an $\ella$-module via the natural algebra homomorphism
$r: \ella \to \drca$. This module is generated by $\Omega_{A|k}^0=A$ since $r(\delta (a))=da$.
One sees that $d\Omega_{A|k}^{*-1} \subseteq \drca$ is an $\ella$-submodule by the relations
\begin{align*}
& \phi (a) \cdot db = a^2 db = d(a^2 b), & q(a) \cdot db = adadb = d(ab)da, \\
& \delta (a) \cdot db = dadb, & \quad u\cdot db=0. 
\end{align*}
\end{remark}

\begin{proposition}
\label{lposa_iso}
There is a natural $\ella$-linear map 
\[ I: \drca \to \ellposa ; \quad a_0 da_1 \dots da_n \mapsto \gamma (a_0) \delta (a_1) \dots \delta (a_n) . \]
One has $\Ker (I) = d\Omega_{A|k}^{*-1}$ such that $\drca /  d\Omega_{A|k}^{*-1} \cong \image (I)$.
\end{proposition}

\begin{proof}
We first verify that $I$ is well-defined. We do have a well-defined map 
\[ I: A\otimes \Lambda^* (A) \to \ellpos (A) \]
since $\gamma (a+b)=\gamma (a)+\gamma (b)$, $\delta (a+b)=\delta (a)+\delta (b)$ and $\delta (a)^2 = 0$.  
In order to see that it factors through $\Omega_{A|k}^*$ it suffices to show that the element
\[ a_0 d(a_1^\prime a_1^{\prime \prime}) da_2 \dots da_n
+a_0 a_1^\prime da_1^{\prime \prime} da_2 \dots da_n + 
a_0a_1^{\prime \prime} da_1^\prime da_2 \dots da_n \]
is mapped to zero. But this is OK by the relation (\ref{rp5}).

Next we see that $I$ is $\ell (A)$-linear. Let $a$, $b$ and $b_1, \dots , b_n$ be elements of $A$. Put 
$\omega = db_1 \dots db_n$ and $\overline \omega = \delta (b_1) \dots \delta (b_n)$. Then we have
\begin{align*}
& I(\phi (a)\cdot b\omega ) = I(a^2b\omega ) = \gamma (a^2 b)\overline \omega =
\phi (a) \gamma(b) \overline \omega = \phi (a) \cdot I(b\omega ), \\
& I(q(a) \cdot b\omega ) = I(ada\medspace b\omega ) = \gamma (ab) \delta (a) \overline \omega = 
q(a) \gamma (b) \overline \omega = q(a) \cdot I(b\omega ), \\
& I(\delta (a) \cdot b\omega ) = I(da\medspace b\omega ) = \delta (a) \gamma (b) \overline \omega =
\delta (a) \cdot I(b\omega ), \\
& I(u\cdot b\omega)=I(0\cdot b\omega ) = 0 =u\cdot I(b\omega).
\end{align*}
where we used the relations (\ref{rp2}), (\ref{rp3}) and (\ref{rp6}).
Having established the $\ella$-linearity, we see that $I$ is simply given by $I(a)=\gamma (a)$ for $a\in A$.

Regarding the kernel, note that $d\Omega_{A|k}^{*-1} \subseteq \Ker (I)$ since $\delta (a)\gamma (1) =0$
by relation (\ref{rp4}). So we have a surjective homomorphism 
$\overline I : \drca / d\Omega_{A|k}^{*-1} \to \image (I)$. 
It suffices to show that it has an inverse $J$, which is necessarily given by $J(\gamma (a))=a$ for $a\in A$.

The submodule $\image (I)$ is the free $\ella$-module generated by $\gamma (a)$ for $a\in A$
modulo the relations (\ref{rp1})--(\ref{rp6}). We must check that $J$ respects these relations. 
This is done by direct verification:
\begin{align*}
& J(\gamma (a+b))=a+b=J(\gamma (a)+\gamma (b)), \\
& J(\phi (a) \gamma (b)) = \phi (a) \cdot b =a^2 b = J(\gamma (a^2b)), \\
& J(q(a)\gamma (b)) = q(a)\cdot b = ada\medspace b = \delta (a) \cdot ab = J(\delta (a)\gamma (ab)), \\
& J(\delta (a)\gamma (b))= \delta (a)\cdot b = da\medspace b \equiv a\medspace db  = a\cdot \delta (b) = 
J(\gamma (a) \delta (b)),\\
& J(\gamma (a) \delta (bc)+\gamma (ab)\delta (c)+\gamma (ac) \delta (b)) = 
a\medspace d(bc)+ab \medspace d(c)+ac\medspace d(b)=0,\\
& J(u\gamma (a))=u\cdot a = 0.
\end{align*}
Thus $\overline I$ is an isomorphism.
\end{proof}

By the definition of the functor $\ellpos$ we directly get the following result:
\begin{proposition}
There is a natural filtration
\[ 0=F_{-1}\ellposa \subseteq F_{0} \ellposa \subseteq F_1 \ellposa \subseteq \dots \subseteq \ellposa \]
where the $\ella$-submodules for $s\geq 0$ are defined by
\[ F_s\ellposa = \ella \langle \gamma (a), \medspace v^i \medspace | \quad a\in A, \medspace i \leq s \rangle .\] 
The subquotients are given by
\[ F_0\ellposa = \image (I) \text{ and } 
F_s\ellposa / F_{s-1}\ellposa = \widetilde \Omega (A) \langle v^s \rangle , \quad s\geq 1
\]
where $\widetilde \Omega (A) = \ella / (I_\delta (A)+(u))$ (which is isomorphic to $\drca$ if one disregards the 
grading).
\end{proposition}

We have a chain complex which models Connes' long exact sequence.
\begin{proposition}
\label{ellConnes}
There is a natural chain complex of $\ella$-modules 
\[ \xymatrix@C=1.0 cm{
\dots \ar[r] & \drca \ar[r]^-{I} & \ellposa \ar[r]^-{\cdot u} & \ellposa \ar[r]^-{D} & \Omega_{A|k}^{*-1} \ar[r] & \dots  
} \]
where the map $D$ is defined by $D(\gamma (a))=da$ and $D(v^i)=0$. At the de Rham modules one has 
$\Ker (I) = \image (D)$. Furthermore, $D\circ I=d$.
\end{proposition}

\begin{proof}
First we see that $D$ is well-defined and $\ella$-linear. We must check that each of the relations
(\ref{rp1})-(\ref{rp9}) are mapped to valid relations in the de Rham complex. The corresponding
valid relations for (\ref{rp1})-(\ref{rp5}) are 
$d(a+b)=da+db$, $a^2db=d(a^2 b)$, $adadb=da\medspace d(ab)$, $dadb=dbda$ and
$da\medspace d(bc) +db\medspace d(ac) + dc\medspace d(ab) = 0$. The remaining corresponding 
relations are trivial.

Next we check that the relevant composites are trivial on module generators such that we have a chain complex.
Firstly, $u \cdot I(a) = u\gamma (a) = 0$ by (\ref{rp6}). Secondly,    
$D(u\gamma (a))=D(0)=0$ and $D(uv^i)=D(v^{i-1})=0$ for $i\geq 1$ by (\ref{rp8}). Thirdly,  
$I(D(\gamma (a)))= I(da)= 0$ and $I(D(v^i))=I(0)=0$. The partial exactness statement follows by Proposition \ref{lposa_iso} above. Finally,
\[ D\circ I (a_0da_1\dots da_n)=D(\gamma (a_0)\delta (a_1) \dots \delta (a_n))= da_0 da_1 \dots da_n. \]
\end{proof}

The chain complex becomes exact when the Cartier map is an isomorphism as we will now show.

\begin{lemma}
\label{starstar}
Assume that the Cartier map $\Phi : \drca \to H_{DR}(A)$ is an isomorphism.
Let $\vnulla$ be the $\ella$-module generated by the symbol $v^0$ subject to the relations 
$uv^0=0$ and $\delta(a)v^0=0$. Then the map
$\vnulla \to \ellposa$ given by  $v^0\mapsto v^0$ is injective.
\end{lemma}

\begin{proof}
We would like to decompose $\ellposa$ into simpler $\ella$-submodules but the last relation mixes the two different
types of generators. So let $M$ be the $\ella$-module with the same generators as $\ellposa$
satisfying the same relations \emph{except} for (\ref{rp9}), which we replace
by the relation $uv^0=0$.  

There is a surjective map $M\to \ellposa$ and a decomposition of $\ell (A)$-modules 
$M =M_\gamma \oplus M_v$ where $M_\gamma$ is the submodule generated by $\gamma(a)$ for $a\in A$,
and $M_v$ is the submodule generated by $v^i$ for $i\geq 0$. 

We will now show that there are injective $\ella$-linear maps 
$i_v : \vnulla \to M_v$ defined by $i_v(v^0) = v^0$ and 
$i_\gamma : \vnulla \to M_\gamma$ defined by $i_\gamma (v^0) = \gamma (1)$.

By Theorem \ref{theorem_deformation} we see that $\vnulla \cong \drca$ and that
an element of $M_v$ can be written in a unique way as a finite sum 
$\sum_{i\geq 0} z_i v^i$ where $z_i\in \drca$. So we can define a left inverse 
(which is $k$-linear, but not $\ella$-linear) by $\sum_{i\geq 0} z_iv^i \mapsto z_0v^0$.
Thus $i_v$ is injective.

The map $i_\gamma$ is well-defined since $u\gamma (1)=0$ 
and $\delta(a)\gamma(1)=\delta(1)\gamma(a)=0$ because $\delta (1) = 0$. 

Recall that $\drca$ is an $\ella$-module via the map $r$ from Proposition \ref{proposition_r} and that 
the boundaries form a submodule $d\Omega_{A|k}^{*-1} \subset \drca$. 
We define an $\ella$-linear map $C:M_\gamma \to \drca /d\Omega_{A|k}^*$ by 
$C(\gamma(a))=a$. It is well-defined since
\begin{align*}
& C(\gamma (a+b) +\gamma (a)+\gamma (b))=a+b+a+b=0,\\
& C(\phi (a)\gamma (b)+\gamma (a^2b))=a^2b+a^2b=0,\\
& C(q(a)\gamma (b)+\delta (a)\gamma (ab))=(a da) b + (da)ab=0,\\
& C(\delta (a)\gamma (b)+\delta (b)\gamma (a))=(da) b + adb = d(ab) \equiv 0 , \\
& C(\gamma (a) \delta (bc)+\gamma (ab) \delta (c) + \gamma (ac) \delta (b))
=a\medspace d(bc)+ab\medspace d(c)+ac\medspace d(b)=0.
\end{align*}

We have $C\circ i_\gamma (v^0)=1$. In particular, the image of $C\circ i_\gamma$ is contained in the de Rham cycles so
it restricts to a map $C^\prime :\vnulla \to H_{DR}(A)$. Using $\ella$-linearity, we see that $C'(q(a)v_0)=[ada]$ and $C'(\phi (a) v_0)=[a^2]$. So $C^\prime$ corresponds to the Cartier map, and is hence an isomorphism. Thus, $i_\gamma$ is injective.

There is a short exact sequence
\[
0 \to \vnulla \xrightarrow{(i_\gamma,-i_v)} M_\gamma \oplus M_v \to \ellposa \to 0.
\]
To see this, note that the only new relation as we pass from 
$M\cong M_\gamma \oplus M_v$ to $\ellposa$ is $\gamma (1) = v^0$.
It follows from the exactness of the sequence together with the
injectivity of $i_v$ that the map $M_\gamma \to \ellposa$ is injective. So the composite 
$\vnulla \to M_\gamma \to \ellposa$, which maps $v^0$ to $v^0$, is also injective.
\end{proof}

\begin{theorem}
If the Cartier map $\Phi : \drca \to H_{DR}(A)$ is an isomorphism then the sequence  
\[ \xymatrix@C=1.0 cm{
\dots \ar[r] & \drca \ar[r]^-{I} & \ellposa \ar[r]^-{\cdot u} & \ellposa \ar[r]^-{D} & \Omega_{A|k}^{*-1} \ar[r] & \dots  
} \]
is long exact.
\end{theorem}

\begin{proof}
The boundaries $\image (d) \subseteq \drca$ form an $\ella$-submodule. Since $D$ maps the generators $\gamma (a)$
and $v^i$ into this submodule we see that $\image (D) \subseteq \image (d)$. 
The restriction $D: \ellposa \to \image (d)$ factors over a map $\bar D: \ellposa /u\ellposa \to \image (d)$ since we
have a chain complex. 
We claim that $\bar D$ is an isomorphism. 

We are going to define an inverse map. As a first step, consider the composite map 
$G:\Omega_{A|k} \xrightarrow{I} \ellposa \to \ellposa/u\ellposa$.
We claim that we can factor $G$  as 
\[ \xymatrix@C=1.0 cm{
G:\Omega_{A|k}^* \ar[r]^-{\pi} & \Omega_{A|k}^* /\Ker (d) \ar[r]^-{\bar G} & \ellposa /u\ellposa .
} \]

This factorization happens if and only if $G(z)=0$ for every $z\in \Ker(d)$. 
Because we are assuming that the Cartier map is an isomorphism, so is 
$\overline r: \Lfa \to \Ker (d)$ by Theorem \ref{theorem_ker(d)}. Thus, the ring $\Ker (d)$ is generated as an 
algebra by elements of the form $da$, $ada$ and $a^2$. So by the definition of $G$ it suffices for us to see that $G(da)=G(ada)=G(a^2)=0$.

We now note the relation $\gamma(1)=v^0=uv^1=0\in \ellposa/u\ellposa$.
The vanishing we want follows from this:  
$G(da)=\delta(a)\gamma(1)=0$,
$G(ada)=\delta(a)\gamma(a)=q(a)\gamma(1)=0$ and 
$G(a^2)= \gamma(a^2)=\gamma(a^2\cdot 1)=\phi(a)\gamma(1)=0$.
From the map $\bar G$ we obtain a map $\image (d) \to \ellposa /u\ellposa$, which is an inverse of $\bar D$. 

We are now ready to prove that the sequence of the lemma is exact. 
We have to check exactness at three modules.
 
\emph{Exactness at the source of $D$:}
This follows since $\bar D$ is an isomorphism.

\emph{Exactness at $\Omega_{A| k}^*$:} 
This is part of Proposition \ref{ellConnes}.

\emph{Exactness at the target of  $I$:} 
We need to prove that the map 
\[ \xymatrix@C=1.0 cm{
U: \ellposa /\image (I) \ar[r]^-{\cdot u} & \ellposa
} \]
is injective. Let $\pi: \ellposa \to \ellposa/\image(I)$ be the canonical projection. We write $w^i=\pi (v^i)$.
The image of $I$ is the $\ella$-submodule of $\ellposa$ generated by the classes $\gamma (a)$,
so that $\ellposa /\image (I)$ is the $\ella$-module generated by
the classes $w^i$ for $i\geq 0$, satisfying the relations
$\delta (a)w^i=0$, $uw^i=w^{i-1}$ and $w^0=0$. That is, 
$\ellposa /\image (I)$ is the $\tilde \ell (A)=\ella /I_\delta(A)$-module generated by
$w^1, w^2, \dots$ subject to the relations $uw^i=w^{i-1}$ for $i\geq 2$ and $uw^1=0$.

It follows from Theorem \ref{theorem_deformation} that an element in 
$\ellposa /\image (I)$ of homogeneous homological degree can be written in a unique way as a sum 
$\sum_{i=1}^{j} z_i w^i$ where $z_i\in \drca$. The composite map $\pi \circ U$
sends this sum to $\sum_{i=2}^{j} z_i w^{i-1}$. So the kernel of $U$ is contained in the submodule 
$(\ellposa/\image(I))_1$ generated by $w^1$. This means that it is enough to show that the restriction
$U: (\ellposa/\image(I))_1 \to \ellposa$ is injective.

There is a surjective map 
$W:\vnulla \to (\ellposa /\image (I))_1$, given on generators by $W(v^0)=w^1$.
Now consider the composite
\[ \xymatrix@C=1.0 cm{
\vnulla \ar[r]^-{W} & (\ellposa /\image (I))_1 \ar[r]^-{U} & \ellposa .
} \]  
The generator $v^0$ goes to $v^0$, so this map agrees with the map 
of Lemma \ref{starstar}. The lemma says that this map is injective so we have the desired result. 
\end{proof}

Both $HH_*(A)$ and $HC_*(A)$ are $HC_*^{-}(A)$-modules as described in section \ref{sec:CHT}.
Via the natural algebra map $\ell (A) \to HC_*^{-}(A)$ they become $\ell (A)$-modules. The functor
$\ellpos$ is an approximation to cyclic homology in the following sense:

\begin{theorem}
There is a natural $\ell (A)$-linear map
\begin{align*} 
\psi^{+} : \ellpos (A) \to HC_*(A); \quad & \gamma (a) \mapsto a\in A= HC_0(A), \\ 
& v^i \mapsto u^{-i} \in HC_{2i}(A)
\end{align*}
which fits into a natural commutative diagram of $\ell (A)$-modules, where the lower row is Connes' exact 
sequence
\[
\xymatrix@C=1.0 cm{
\dots \ar[r] & \drca \ar[r]^-{I} \ar[d]^-{\epsilon} & \ellposa \ar[r]^-{\cdot u} \ar[d]^-{\psi^+} 
& \ellposa \ar[r]^-{D} \ar[d]^-{\psi^+} & \Omega_{A|k}^{*-1} \ar[r] \ar[d]^-{\epsilon} & \dots \\ 
\dots \ar[r] & HH_*(A) \ar[r]^-{I} & HC_*(A) \ar[r]^-{\cdot u} 
& HC_{*-2}(A) \ar[r]^-{\partial} & HH_{*-1}(A) \ar[r] & \dots
}\] 
If both $\epsilon : \drca \to HH_*(A)$ and the Cartier map $\Psi : \drca \to H_{DR}(A)$ are isomorphisms 
then $\psi^+$ is an isomorphism. In consequence $\psi^+$ is an isomorphism when $A$ is 
a smooth algebra.
\end{theorem}

\begin{proof}
We must verify that (\ref{rp1})-(\ref{rp9}) are mapped to valid relations in cyclic homology. 
This is obvious for the relations (\ref{rp1}), (\ref{rp6}), (\ref{rp8}) and (\ref{rp9}). 

For (\ref{rp2}) we get the desired result since
\[ \mu^+ (1\otimes a^2[]+ u\otimes 1[a|a], 1\otimes b[]) = \mu^+ (1\otimes a^2[], 1\otimes b[]) =
1\otimes a^2b[]. \]

For (\ref{rp3}) we use the calculation
\[ \mu^+ (1\otimes a[a], 1\otimes b[]) = 1\otimes ab[a] = \mu^+ (1\otimes 1[a], 1\otimes ab[]). \]

For (\ref{rp4}) we must verify that $1\otimes (b[a]+a[b])$ is a boundary. It equals the boundary
$\partial^+ (1\otimes 1[a|b]+u^{-1}\otimes ab[])$.

For (\ref{rp5}) we must see that $1\otimes (a[bc]+ab[c]+ac[b])$ is a boundary. It equals
$\partial^+ (1\otimes a[b|c])$.

Finally, for (\ref{rp7}) we have $\mu^+ (1\otimes 1[a], u^{-i}\otimes 1[] )= u^{-i} \otimes 1[a]$ which
equals the boundary $\partial^+ (u^{-(i+1)}\otimes a[])$. 

Thus, we have a well-defined map $\psi^+$ as stated which is $\ell (A)$-linear by the definition 
of the $\ell (A)$-action on $HC_*(A)$. 

It suffices to check the commutativity of the diagram on $\ell (A)$-module generators. Both composites in the square 
to the left maps $a\in \Omega_{A|k}^0$ to $a\in HC_0(A)$. In the middle square both composites 
maps $\gamma (a)$ to zero and $v^i$ to $u^{-i+1}$ for $i\geq 1$. In the square to the right, recall 
that $B_*:HH_{*-2}(A)\to HH_{*-1}(A)$ factors through $HC_{*-2}(A)$ as $B_*=\partial \circ I$. We have 
$\psi^+ (\gamma (a)) =  1\otimes a[]= I(a[])$ such that $\partial (\psi^+ (\gamma (a))) = B_*(a[]) = 1[a]$.
But also $\epsilon (D(\gamma (a)))=\epsilon (da) = 1[a]$. The generator $v^i$ is mapped to zero by both composites.

If $\Psi$ is an isomorphism then the upper row of the diagram is also exact. 
The homological degrees of the modules are bounded below and the isomorphism statement follows by the 5-lemma.
\end{proof}

For periodic cyclic homology the situation is easier.
\begin{definition}
\label{def:ellper}
Let $\ellper (A)$ be the free commutative and unital $k$-algebra on generators 
$\phi (a)$, $q(a)$ for $a\in A$ and $u$, $u^{-1}$ modulo the following relations:
\begin{align}
& \phi (a+b) = \phi (a)+\phi (b), \\
& q(a+b) = q(a)+q(b), \\
& \phi (ab) = \phi (a) \phi (b)+uq(a)q(b), \\
& q(ab) = q(a)\phi (b)+ \phi (a) q(b), \\
& q(a)^2 = 0, \\
& uu^{-1} = 1, \\
& \phi (1) = 1. \label{rper7}
\end{align}
The upper degrees of the generators are $|\phi (a)| = 2|a|$, $|q(a)| = 2|a|-1$, $|u| = 2$ and $|u^{-1}| = -2$.
The homological degrees are $||\phi (a)|| = 0$, $||q(a)||=1$, $||u||=-2$ and $||u^{-1}||=2$.
\end{definition}

We have a chain complex which models the long exact sequence for periodic cyclic homology.
\begin{proposition}
There is a natural chain complex
\[
\xymatrix@C=1.0 cm{
\dots \ar[r] & \ella \ar[r]^-{\iota} & \ellpera \ar[r]^-{S} & \ellposa \ar[r]^-{\partial} & \ella \ar[r] & \dots  
}\] 
Here $\iota$ is the $k$-algebra homomorphism defined by 
\[ \delta (a) \mapsto 0, \quad \phi (a)\mapsto \phi (a), \quad q(a)\mapsto q(a), \quad u\mapsto u. \] 
Via this map $\ellpera$ becomes an $\ella$-module generated by $u^{-i}$ for $i\geq 1$. The map 
$S$ is $\ella$-linear and given by $S(u^{-i})=v^{i-1}$. 
The map $\partial$ is $\ella$-linear and given by $\partial (\gamma (a)) = \delta (a)$ and 
$\partial (v^i)=0$.
\end{proposition}

\begin{proof}
The defining relations for $\ell (A)$ are obviously respected such that $\iota$ becomes a well-defined
algebra homomorphism. It remains to verify that the relevant composites are trivial on module generators. 
For $\iota \circ \partial$ and $\partial \circ S$ this follows directly.
For $S\circ \iota$ one has $S(\iota (1))=S(1)=S(uu^{-1})=uv^0=u\gamma (1)=0$.
\end{proof}

For periodic cyclic homology the approximation theorem looks as follows:
\begin{theorem}
There is a natural algebra homomorphism
\begin{align*} 
\psi^{per} : \ellper (A) \to HC^{per}_*(A); \quad & q(a) \mapsto 1\otimes a[a], \\ 
& \phi (a) \mapsto 1\otimes a^2[] +u \otimes 1[a|a], \\
& u\mapsto u\otimes 1[], \\
& u^{-1} \mapsto u^{-1} \otimes 1[]
\end{align*}
which fits into a natural commutative diagram
\[
\xymatrix@C=1.0 cm{
\dots \ar[r] & \ella \ar[r]^-{\iota} \ar[d]^-{\psi} & \ellpera \ar[r]^-{S} \ar[d]^-{\psi^{per}} 
& \ellposa \ar[r]^-{\partial} \ar[d]^-{\psi^+} & \ella \ar[r] \ar[d]^-{\psi} & \dots \\
\dots \ar[r] & HC_*^-(A) \ar[r]^-{\iota} & HC_*^{per}(A) \ar[r]^-{S} 
& HC_{*-2}(A) \ar[r]^-{\partial} & HC_{*-1}^-(A) \ar[r] & \dots 
}\] 
If $A$ is smooth over $k$, supplemented and of finite type, then $\psi^{per}$ is an isomorphism.
\end{theorem}

\begin{proof}
The image classes of the first types of generators $q(a)$, $\phi (a)$ and $u$ are cycles of the complex 
$T^{-\infty , \infty }_* (A)$ since we have proven earlier that they are cycles of the subcomplex $T^{-\infty , 0}_*(A)$.
Consider the chain
\[ u^{-1}\otimes a[] = a \in \overline B_{1,1} (A) = A. \]
We have $b(a[])=0$ and $B(a[])=1[a]$ such that $\partial^{per} (u^{-1} \otimes a[]) = 1\otimes 1[a].$
If we take $a=1$ we see that the image class of $u^{-1}$ is a cycle as well. For general $a$ we find that
$\iota (\psi (\delta (a)))=0$. It follows, that the first five relations of $\ellper (A)$ are respected by $\psi^{per}$
since they appear as relations of $\ell (A)/I_\delta (A)$. The same thing happens with the last relation. 
The definition of the product structure on $HC^{per}_*(A)$ gives directly that the second last relation of 
$\ellper (A)$ is respected. 

The algebra homomorphisms $\psi^{per}\circ \iota$ and $\iota \circ \psi$ agree on the generators of $\ell (A)$ 
so the square to the left commutes. The $\ell (A)$-linear maps $\psi^+ \circ S$ and $S\circ \psi^{per}$ 
agrees on the module generators $u^{-i}$, $i\geq 1$ so the middle square commutes. 
In order to show that the composites in the square to the right agrees on module generators $\gamma (a)$ and $v^i$, 
one must show that the lower boundary map satisfies $\partial (a) = 1\otimes 1[a]$ for $a\in A$ and 
$\partial (u^{-i})=0$. This follows by the snake lemma. The chain $u^{-1}\otimes a[]\in T_*^{-\infty , \infty }$ 
is a lift of the cycle $a\in \overline B_{1,1} (A)\subseteq T_*^{1, \infty}$ and 
$\partial^{per}(u^{-1}\otimes a[])=1\otimes 1[a] = \iota (1\otimes 1[a])$. The chain $u^{-i}\otimes 1[]$ is a lift of 
the cycle $u^{-i}\otimes 1[] \in T_*^{1,\infty}$ and $\partial^{per}(u^{-i}\otimes 1[])=0$.

In general, the map $\iota : HC^-_n (A) \to HC^{per}_n (A)$ is an isomorphism for $n<0$. 
The relation $u\delta (a)=0$ ensures that $\iota : \ella \to \ellpera$ is also an isomorphism in negative 
homological degrees. Under the extra assumption on $A$, the map $\psi$ is an isomorphism 
by our main theorem. It then follow by the square to the left in the diagram
that $\psi^{per}_*$ is an isomorphism in negative homological degrees. 
By periodicity, it is an isomorphism in all degrees.
\end{proof}

\section{Simplicial algebras and spectral sequences}
\label{sec:spectral_sequences}

In this section we set up  spectral sequences relating the non-abelian derived functors of $\ell$, $\ellpos$, $\ellper$
to the cyclic homology theories. We start by reviewing Goodwillie's definition of of these homologies for simplicial
algebras \cite{Goodwillie} page 363.

Let $k$ be a commutative and unital ring. Assume that we have a cyclic $k$-algebra
\[ X :\scat C^{op} \to \Alg \]
The main example comes from the cyclic bar construction $X=Z(R)$ for a $k$-algebra $R$. 
Here $X_n = R\otimes R^{\otimes n}$.

Form the $\ZZ \times \ZZ$-graded bicomplex $(\overline B_{**} (X) , B, b)$ with
\[ \overline B_{p,q} (X) = \begin{cases}
\overline X_{q-p} , & p \leq q ,\\
0, & p>q
\end{cases} \]
where $\overline X_n = X_n /\text{degeneracies}$. 
The horizontal boundary map is a "Connes' boundary" $B$ of bidegree $(-1,0)$. The vertical boundary map is a
"Hochschild boundary" of bidegree $(0,-1)$. See \cite{Goodwillie} for the precise definitions.

For $-\infty \leq \alpha \leq \beta \leq \infty$ we let $T^{\alpha , \beta }_* (X)$ be the $\ZZ$-graded chain complex with
\[ T^{\alpha , \beta }_n (X) = \prod_{\alpha \leq p \leq \beta} \overline B_{p, n-p} (X) \]
and boundary map $B+b$. Then
\begin{center}
\begin{tabular}{l l}
$HH_*(X) = H_* (T^{0,0}_* (X)), \quad $ & $HC_*(X) = H_*(T^{0, \infty }_* (X)),$ \\
$HC^-_* (X) = H_*(T^{-\infty , 0}_* (X)), \quad $ & $HC^{per}_*(X) = H_*(T^{-\infty , \infty }_* (X)).$
\end{tabular}
\end{center}
When $X=Z(R)$ we get the homology theories of $R$ described in Definition \ref{def:homologies}.

For a simplicial $k$-algebra $R_\bullet$ one has a simplicial cyclic $k$-algebra, or equivalently a functor
\[ Z(R_\bullet ) : \scat C^{op} \times \scat^{op} \to \Alg . \]
where $Z(R_\bullet )_{n,m} = R_m \otimes R_m^{\otimes n}$.
So assume we have a functor 
\[ {\bf X} : \scat C^{op} \times \scat^{op} \to \Alg .\]
Form the $\ZZ \times \ZZ \times \ZZ$-graded triple complex $(\overline B_{***} ({\bf X}), B, b, d)$ with
\[ \overline B_{p, q, r} (\bf X) = \begin{cases}
\overline {\bf X}_{q-p,r} , & p\leq q , \quad 0 \leq r, \\
0, & \text{ otherwise.} 
\end{cases} \]
The boundary maps are Connes' boundary $B$ of triple-degree $(-1, 0, 0)$, Hochschild's boundary $b$ 
of triple-degree $(0, -1, 0)$ and the boundary $d$ of triple-degree $(0, 0, -1)$ which comes from the simplicial maps 
$d_i: {\bf X}_{n,m} \to {\bf X}_{n,m-1}$.

We form the $\ZZ$-graded chain complex $T^{\alpha , \beta}_* ({\bf X})$ with
\[ T^{\alpha , \beta }_n ({\bf X}) = \prod_{\alpha \leq p \leq \beta } \medspace
\bigoplus_{p+q+r=n} \overline B_{p,q,r}({\bf X}) \]
and boundary map $B+b+d$. Then 
\begin{center}
\begin{tabular}{l l}
$HH_*({\bf X}) = H_* (T^{0,0}_* ({\bf X})), \quad$ &
$HC_*({\bf X}) = H_*(T^{0,\infty }_* ({\bf X)}),$ \\
$HC^-_* ({\bf X}) = H_*(T^{-\infty , 0}_* ({\bf X})), \quad $ &
$HC^{per}_*({\bf X}) = H_*(T^{-\infty , \infty }_* ({\bf X})).$
\end{tabular}
\end{center}
When ${\bf X} = Z(R_\bullet )$ we get the various homologies of the {\em simplicial} algebra $R_\bullet$.

Finally, Goodwillie proves, \cite{Goodwillie} Lemma I.3.5,  that all four homology theories of simplicial algebras are 
homotopy invariant: A weak equivalence between two simplicial algebras, which are both flat over $k$, induces 
isomorphisms.  

We will now define a filtration of the chain complex $T^{\alpha , \beta }_* ({\bf X})$. First we introduce some 
sets of integers.
\begin{definition}
For integers $n$, $p$ and $s$ we let
\begin{align*} 
& J_n(p) = \{ (q,r) \in \ZZ \times \ZZ | \quad p \leq q, \quad 0 \leq r, \quad p+q+r=n \} , \\
& J_n^s(p) = \{ (q,r) \in J_n(p) | \quad r \leq s\} .
\end{align*}
\end{definition}

Note that $J_n(p)$ is a finite set for all $n$ and $p$. Furthermore, we have a filtration of {\em finite} length
\[ \emptyset = J_n^{-1}(p) \subseteq J_n^0(p) \subseteq J_n^1(p) \subseteq \dots \subseteq J_n(p) . \]
Note also that
\[ T^{\alpha , \beta}_n ({\bf X}) = 
\prod_{\alpha \leq p \leq \beta } \medspace \bigoplus_{(q,r) \in J_n(p)} {\overline {\bf X}}_{q-p,r} . \]

\begin{definition}
Let  
\[ 0 = F_{-1}T_*^{\alpha, \beta }({\bf X}) \subseteq  F_{0}T_*^{\alpha, \beta }({\bf X}) \subseteq
F_{1}T_*^{\alpha, \beta }({\bf X}) \subseteq \dots \subseteq T_*^{\alpha, \beta }({\bf X}) \]
be the filtration of chain complexes defined by
\[ F_s T^{\alpha , \beta}_n ({\bf X}) = 
\prod_{\alpha \leq p \leq \beta } \medspace \bigoplus_{(q,r) \in J_n^s (p)} {\overline {\bf X}}_{q-p,r} . \]
Note that the length of this filtration if finite in each degree.
\end{definition}

The filtration quotients becomes
\begin{align*} 
 F_s T^{\alpha , \beta}_n ({\bf X}) / F_{s-1} T^{\alpha , \beta}_n ({\bf X}) & = 
\prod_{\alpha \leq p \leq \beta } \medspace \bigoplus_{(q,r) \in J_n^s (p) \setminus J_n^{s-1}(p)} 
{\overline {\bf X}}_{q-p,r} \\
&=  \prod_{\alpha \leq p \leq \beta } \medspace \bigoplus_{\overset {p\leq q} {p+q = n-s}} 
{\overline {\bf X}}_{q-p,s} \\
&= T^{\alpha , \beta }_{n-s}({\bf X}_{\bullet , s} ). 
\end{align*}

The filtration gives us a spectral sequence with $E^1$-page
\begin{align*} E^1_{s,t}(\alpha , \beta ) &= 
H_{s+t} \big( F_s T^{\alpha , \beta}_*({\bf X}) / F_{s-1} T^{\alpha , \beta}_*({\bf X}) \big) \\
&= H_{t} (T_*^{\alpha , \beta }({\bf X}_{\bullet , s}), B+b)
\end{align*}
and $E^2$-page
\[ E^2_{s,t}(\alpha , \beta ) = H_s( H_{t} (T_*^{\alpha , \beta }({\bf X}_{\bullet , *}), B+b), d). \]
Since the filtration is finite in each degree, the spectral sequence converges strongly to the homology
$H_*( T^{\alpha , \beta }_* ({\bf X}), B+b+d)$. Thus in the case $X=Z(R_\bullet )$ we have strongly 
convergent spectral sequences
\begin{align*}
& E^1_{s,t}(0, 0) = HH_t(R_s) \Rightarrow HH_*(R_\bullet ), \\
& E^1_{s,t}(-\infty ,0 ) = HC^-_t(R_s) \Rightarrow HC_*^- (R_\bullet ), \\
& E^1_{s,t}(0, \infty ) = HC_t(R_s) \Rightarrow HC_*(R_\bullet ), \\
& E^1_{s,t}(-\infty , \infty ) = HC^{per}_t (R_s) \Rightarrow HC^{per}_* (R_\bullet ).
\end{align*}

At least for commutative algebras we can interpret these spectral sequences further.
Let $\Commalg$ denote the category of commutative $k$-algebras with unit. As remarked already by Quillen in
\cite{Quillen2} page 66, section 1 one can define left derived functors of {\em any} functor $F: \Commalg \to {\cal A}$ 
where ${\cal A}$ is an abelian category. The definition uses simplicial resolutions which give much more computational
freedom than one has in the cotriple derived functor setting. We will go through the construction of these derived functors.

The category $\Commalg$ is complete and cocomplete. The initial object is $k$, the terminal object is the zero algebra, 
the product is the Cartesian product of underlying sets with addition and multiplication defined component-wise. 
The coproduct of two objects is the tensor product $R\sqcup S = R\otimes S$.

Write $s{\cal C}$ for the category of simplicial objects in a category ${\cal C}$.
We equip $s\Commalg$ with its standard simplicial model category structure \cite {Quillen1}. 
It is described in \cite{G-J} section 4 page 97. The setting is more general there, but simplicial commutative 
algebras are included according to the example on page 103. 

The polynomial algebra functor is left adjoint of the forgetful functor 
\[ k[-]: \Set \rightleftarrows \Commalg : U .\] 
Applying this degree-wise, we get an adjoint pair $k[-]: s\Set \rightleftarrows s\Commalg : U$.
A morphism $f$ in $s\Commalg$ is a weak equivalence if $U(f)$ is a weak equivalence, a fibration if
$U(f)$ is a fibration and a cofibration if it has the left lifting property with respect to all trivial fibrations in 
$s\Commalg$. 

For objects $A$ and $B$ in $s\Commalg$ and $K$ in $s\Set$, we have a co-power operation 
$A\otimes K\in s\Commalg$ which is described on page 85 and 100 of \cite{G-J}. It is given by
\[ (A\otimes K)_n = \bigsqcup_{k\in K_n} A_n . \]
The simplicial mapping space $\Map (A, B)$ is given by
\[ \Map (A,B)_n = \Hom (A\otimes \Delta^n , B) \] 
where $\Delta^n$ is the simplicial set $\Hom_\Delta (-,[n])$ and the simplicial maps $d_i$ and $s_i$ are induced by the 
cosimplicial maps $d^i : \Delta^{n-1} \to \Delta^n$ and $s^i :\Delta^{n+1}\to \Delta^n$.

A simplicial homotopy from $f:A\to B$ to $g:A\to B$ is a map $h\in \Map (A,B)_1$ 
such that $d_0h = f$ and $d_1h = g$.

If we have a commutative square
\begin{equation}
\label{llp}
\xymatrix@C=1.0 cm{
A^\prime \ar[r] \ar@{>->}[d]_i & A \ar@{>>}[d] ^-{p}_-{\sim} \\
B^\prime \ar[r] \ar@{-->}[ru]^-{\hat p} & B
}
\end{equation}
where $A^\prime$ is cofibrant, $i$ is a cofibration and $p$ a trivial fibration, then the lifting $\hat p$ 
is unique up to simplicial homotopy under $A^\prime$ and over $B$ by \cite{G-J} Proposition 3.8 page 92.

One has a more general notion of simplicial homotopy \cite{Weibel} 8.3.11. If $f, g:X \to Y$ are simplicial objects in any category ${\cal C}$,
a simplicial homotopy from $f$ to $g$ is a family of morphisms $h_i: X_n \to Y_{n+1}$, $i = 0, \dots , n$ of 
${\cal C}$ such that $d_0h_0 = f_n$ and $d_{n+1} h_n = g_n$, while
\[
d_ih_j = \begin{cases}
h_{j-1}d_i, & i<j, \\
d_ih_{i-1}, & i=j\neq 0, \\
h_jd_{i-1}, & i>j+1,
\end{cases}
\]

\[
s_ih_j = \begin{cases}
h_{j+1}s_i, & i\leq j, \\
h_js_{i-1}, & i>j.
\end{cases}
\]

For ${\cal C}= \Commalg$ there is a one to one correspondence between the two notions of simplicial homotopy. This
follows by the proof of \cite{Weibel} Theorem 8.3.12, which carries through for any category which is finitely cocomplete.
One simply has to replace $A\times \Delta^1$ by $A\otimes \Delta^1$ (as indicated in exercise 8.3.5).

If instead ${\cal C}$ is an abelian category ${\cal A}$ and $f,g: V \to W$ are two simplicially homotopic maps, then
$f_*, g_* : N(V) \to N(W)$ are chain homotopic maps between the normalized chain complexes by \cite{Weibel} Lemma 8.3.13. 

A commutative $k$-algebra $A$ is viewed as an object of $s\Commalg$ as the constant simplicial algebra, where one has a copy of $A$ in each degree and all face and degeneracy maps are identity maps. A simplicial resolution of $A$ is a trivial 
fibration $R\to A$ where $R$ is cofibrant. Simplicial resolutions exist since one can factor the unit $k\to A$ into a 
cofibration followed by a trivial fibration $\xymatrix@C=0.5 cm{ k \ar@{>->}[r] & R \ar@{>>}[r]^-{\sim} & A }$.

If $F:\Commalg \to \Mod$ is {\em any } functor from commutative $k$-algebras to $k$-modules, 
then we can apply $F$ in each simplicial degree of a resolution $R\to A$ and define the left derived functors 
as the homology groups of this object $L_iF (A) = \pi_i F(R)$, $i=1,2, \dots$. The definition is independent 
of the choice of resolution and functorial in $A$ as one sees by the diagram (\ref{llp}) and its homotopy 
uniqueness property applied to the two settings
\[
\xymatrix@C=1.0 cm{
k \ar@{>->}[r] \ar@{>->}[d] & R \ar@{-->}@<0.5ex>[ld] \ar@{>>}[d]^-{\sim} 
& & k \ar@{>->}[d] \ar[dr] \ar@{=}[r] & k \ar@{>->}[d] \\
R^\prime \ar@{-->}@<0.5ex>[ru] \ar@{>>}[r]_-{\sim} & A 
& & R \ar@{>>}[d]_-{\sim} \ar[dr] \ar@{-->}[r] & S \ar@{>>}[d]^-{\sim} \\
& & & A \ar[r]^-{f} & B
}
\]
Note that $F$ caries a simplicial homotopy to a simplicial homotopy according to the general notion. 

By the derived functors of an endofunctor $F$ on $\Commalg$ we understand the
derived functors of the composite $V\circ F$ where $V: \Commalg \to \Mod$ is the forgetful functor.

For an object $A$ in $\Commalg$ we can form a simplicial resolution $R\to A$ and consider the 
spectral sequences above. The $E^2$-pages becomes derived functors. For the target groups we can use 
homotopy invariance. Thus we have strongly convergent spectral sequences
\begin{align*}
& E^2_{s,t}(0,0) = L_s(HH_t)(A) \Rightarrow HH_*(A), \\
& E^2_{s,t}(-\infty ,0 ) = L_s(HC^-_t)(A) \Rightarrow HC_*^- (A), \\
& E^2_{s,t}(0, +\infty ) = L_s(HC_t)(A) \Rightarrow HC_*(A), \\
& E^2_{s,t}(-\infty , +\infty ) = L_s(HC^{per}_t)(A) \Rightarrow HC^{per}_*(A).
\end{align*}

One can always choose a simplicial resolution, which is a polynomial algebra in each degree. The argument is given in  \cite{Quillen1} Chapter 2, p 4.11 Remark 4 and in \cite{Miller} section 3. Let us review parts of it. 

A morphism $i:R\to S$ in $s\Commalg$ is called {\em almost free} provided that there 
is a sequence of subsets $X_n \subseteq S_n$, $n=0,1, \dots$ such that for each $n$ one has
firstly $s_i X_n \subseteq X_{n+1}$ for $0\leq i \leq n$ and secondly the natural map $R_n \otimes k[X_n] \to S_n$ is an isomorphism. 

One can show that any almost free map is a cofibration. Furthermore, any morphism $R\to S$ in 
$s\Commalg$ admits a factorization $\xymatrix@C=0.5 cm{R \ar@{>->}[r] & S^\prime \ar@{>>}[r]^-{\sim} & S}$
in which $\xymatrix@C=0.5 cm{R \ar@{>->}[r] & S^\prime}$ is almost free and 
$\xymatrix@C=0.5 cm{S^\prime \ar@{>>}[r]^-{\sim} & S}$ is a trivial fibration. By factoring the unit $k\to A$
in this way, one gets an almost free resolution  
$\xymatrix@C=0.5 cm{k \ar@{>->}[r] & A^\prime \ar@{>>}[r]^-{\sim} & A}$ where $A^\prime_n$ is a 
polynomial algebra for each $n$. 

Choose an almost free simplicial resolution of $A$. By the Hochschild-Kostant-Rosenberg theorem we have
natural isomorphisms of functors 
\[ L_i (HH_t) \cong L_i(\Omega_{-|k}^t), \quad i = 0,1, \dots \] 
such that we get the well-known spectral sequence 
\[ E^2_{s,t}(0,0) = L_s(\Omega^t_{-|k}) (A) \Rightarrow HH_*(A). \]

By the remaining spectral sequences and our approximation functors we get the following result:
\begin{theorem}
\label{thm:l-funktor_ss}
Let $A$ be a finitely generated commutative and unital $\FF_2$-algebra (possibly non-negatively graded). 
Then there are natural isomorphisms of non-abelian derived functors for $i=0, 1, 2, \dots$ as follows:
\begin{align*}
& L_i (\ell  )(A) \cong L_i (HC_*^-)(A) , \\
& L_i(\ellpos )(A) \cong L_i(HC_*)(A), \\ 
& L_i(\ellper )(A) \cong L_i(HC_*^{per})(A).
\end{align*}
Furthermore, there are strongly convergent spectral sequences
\begin{align*}
& E^2_{s,*}(-\infty ,0 ) = L_s(\ell )(A) \Rightarrow HC_*^- (A), \\
& E^2_{s,*}(0, +\infty ) = L_s(\ellpos )(A) \Rightarrow HC_*(A), \\
& E^2_{s,*}(-\infty , +\infty ) = L_s(\ellper )(A) \Rightarrow HC^{per}_*(A).
\end{align*}
\end{theorem}

\begin{proof}
By \cite{Weibel} 8.8.3. there exists an almost free resolution $R\to A$ such that each $R_n$ is a 
polynomial algebra on finitely many generators. A polynomial algebra is always supplemented so the
stated isomorphisms follows by our approximation theorems. 
\end{proof}

\begin{remark}
If $A=k[X]$ is a polynomial algebra on finitely many generators (or finitely many generators in each degree), 
then the identity map $A\to A$ of the constant simplicial algebra is an almost free simplicial resolution. 
Thus the higher derived functors become trivial such that $E^2_{s,*}=0$ for $s\geq 1$ in each of the four 
spectral sequences. So the spectral sequences collapse and we get isomorphisms 
$L_0(F)(k[X]) \cong F(k[X])$ for each functor $F=\Omega_{-|k}$, $\ell$, $\ellpos$
and $\ellper$. Hence we recover the approximation theorems for polynomial algebras.
\end{remark}

\section{The zeroth derived functors and universality}

The zeroth derived functor is sometimes given by the following result:
\begin{lemma}
\label{0th}
Let $k$ be a commutative ring and let $\Commalgnu$ denote the category of 
commutative $k$-algebras where we do no longer require that objects are unital. Let 
$F$ be an endofunctor on this category. Assume that for every surjective morphism 
$f:A\to B$ in $\Commalgnu$ the following two conditions hold: 
\begin{enumerate}
\item The map $F(f): F(A) \to F(B)$ is surjective,
\item The sequence $F(\Ker f) \to F(A) \to F(B)$ is exact in the category of $k$-modules.
\end{enumerate}
Then there is an algebra isomorphism $L_0F(C) \cong F(C)$ for all $C$ in $\Commalg$.
\end{lemma}

\begin{proof}
The proof of \cite{BO_Topology} Lemma 6.2 carries through in this setting. But let us comment a little further on 
the multiplicative structure. Let $P\to C$ be a simplicial resolution. From the normalized chain complex
$N_*F(P)$ one sees that 
\[ L_0F(C) = F(P_0)/F(d_1)(\Ker F(d_0)). \] 
By the simplicial identities, $d_1 : P_1 \to P_0$ is surjective. Thus $F(d_1)$ is surjective such that $F(d_1)(\Ker F(d_0))$ 
is an ideal in $F(P_0)$ and $L_0F(C)$ is an algebra. 
\end{proof}

For $k = \FF_2$ we let $\ellnu$ and $\ellpernu$ be the functors which are 
defined as $\ell$ and $\ellper$ except that we exclude the last relation 
(\ref{r12}) and (\ref{rper7}) in each case.

\begin{proposition}
\label{prop:L_0ellnu}
There are natural isomorphisms $L_0(\ellnu ) \cong \ellnu$, $L_0(\ellpernu ) \cong \ellpernu$.
\end{proposition}

\begin{proof}
We use Lemma \ref{0th}. Note that by their definitions $\ellnu$ and $\ellpernu$ extend
to functors on $\Commalgnu$. We verify that the two conditions hold. The argument follows 
the proof of \cite{BO_Topology} Proposition 7.5. 

Let $I \subseteq A$ be an ideal and let $f:A \to A/I$ be
the canonical projection. The map $\ellnu (f)$ is surjective with kernel
\[ J = (\phi (x) -\phi (y), \medspace q(x)-q(y), \medspace \delta (x)-\delta (y) \medspace | \medspace x-y \in I). \]
Thus we must check that $\ellnu (I) = J$. 

The inclusion $\ellnu (I) \subseteq J$ holds since $\phi (0)=q(0)=\delta (0)$. Conversely, we have the
relations $\phi (x)-\phi (y) = \phi (x-y)$, $\delta (x)-\delta (y)=\delta (x-y)$ and, since $\delta (x^2)=0$, also
\[ q(x)-q(y)=q(x-y)+\delta (x(x-y)) . \] 
The other inclusion $\ellnu (I) \supseteq J$ follows. By the lemma above, $L_0(\ellnu )(A) \cong \ellnu (A)$.
A similar argument shows that $L_0(\ellpernu )(A) \cong \ellpernu (A)$.
\end{proof}

Put $w=\phi (1)+1$. We have
\[ \ell (A) = \ellnu (A)/ w \ellnu (A) , \quad \ellper (A) = \ellpernu (A) /w \ellpernu (A) \]
which gives us long exact sequences of derived functors
\[ \xymatrix@C=0.5 cm{\dots \ar[r] & L_*(w\ellnu )(A) \ar[r] & L_*(\ellnu )(A) \ar[r] 
& L_*(\ell )(A) \ar[r]^-{\partial} & L_{*-1}(w\ellnu )(A) \ar[r] & \dots }\]
\[ \xymatrix@C=0.5 cm{\dots \ar[r] & L_*(w\ellpernu )(A) \ar[r] & L_*(\ellpernu )(A) \ar[r] 
& L_*(\ellper )(A) \ar[r]^-{\partial} & L_{*-1}(w\ellpernu )(A) \ar[r] & \dots }\]

\begin{proposition}
\label{prop:0th}
There are natural isomorphisms of derived functors
\[ L_i \ellnu  \cong L_i \ell , \quad L_i \ellpernu  \cong L_i \ellper \]  
for $i\geq 1$. Furthermore for objects $A$ in $\Commalg$ one has natural isomorphisms of algebras
\[ L_0\ell  (A) \cong \ell (A), \quad  L_0 \ellper (A) \cong \ellper (A). \]
\end{proposition}

\begin{proof}
We prove the results for the functor $\ell$. The argument for $\ellper$ is similar.
If we multiply $w$ by one of the generators $\phi (a)$, $q(a)$ or $\delta (a)$ of $\ellnu (A)$ 
we get zero as mentioned in Remark \ref{remark:ell}. Thus the ideal $w\ellnu (A)$ has the form
$wk[u]$ which does not depend on $A$. 

Let $\epsilon: P\to A$ be a simplicial resolution of $A$. Then $w\ellnu (P)$ is a constant simplicial object
such that $L_i(w\ellnu )(A) = 0$ for $i\geq 1$ and $L_0(w\ellnu) (A) = wk[u]$. By the long exact sequence
of derived functors we get isomorphisms as stated for $i\geq 2$.

The kernel of the canonical projection $\ellnu (A) \to \ell (A)$ is the ideal $w\ellnu (A)=wk[u]$. 
One can identify the map 
\[ wk[u] \cong L_0(w \ellnu) (A) \to L_0 (\ellnu ) (A) \cong \ellnu (A) \]
as the inclusion of this kernel into the domain of the projection. So it is an injective map and the remaining 
results follows.
\end{proof}

\begin{proposition}
\label{prop:0th+}
The canonical map $L_0 \ellpos (A) \to \ellpos (A)$ is an isomorphism of $\ella$-modules 
for every object $A$ in $\Commalg$.
\end{proposition}

\begin{proof}
Pick an almost free simplicial resolution $P\to A$. By the description in Remark 
\ref{ellpos_functor} one sees that the action
\[ \ell (P_0) \otimes \ellpos (P_0)\to \ellpos (P_0) \to L_0\ellpos (A) \]
factors through a well-defined action $L_0\ell (A) \otimes L_0\ellpos (A) \to L_0 \ellpos (A)$. So we have an  
action $\ella \otimes L_0 \ellpos (A) \to L_0 \ellpos (A)$ by the proposition above.

Consider the composite $\ellpos (P_0) \to L_0 \ellpos (A) \to \ellpos (A)$ of the projection followed
by the canonical map. The composite equals $\ellpos (\epsilon )$ where $\epsilon : P_0 \to A$ is the map from the 
resolution. A small diagram chase, which uses the surjectivity of $\epsilon$, shows that the canonical map 
is $\ella$-linear.

We construct an inverse of the canonical map. 
Let $r: \ellposa \to L_0 \ellposa$ be the $\ell (A)$-linear map with $r(v^i) = [v^i]$ and
$r(\gamma (a)) = [\gamma (a^\prime )]$, where $a^\prime$ is a lift of $a$ ie. $\epsilon (a^\prime ) = a$.
We must check that $r$ is well-defined. By the additivity of $\gamma$ we see that it is independent of
the choice of lift: Assume that $a^{\prime \prime}$ also satisfies $\epsilon (a^{\prime \prime})= a$. 
Then $\epsilon (a^\prime -a^{\prime \prime}) = 0$. By the exactness 
of the resolution, there exists a $z\in R_1$ such that $d_0 z = 0$ and $d_1z = a^\prime -a^{\prime \prime}$.
Thus,
\[ \gamma (a^\prime )-\gamma (a^{\prime \prime}) = \gamma (a^\prime -a^{\prime \prime}) =
\gamma (d_1z) = \ellpos (d_1)(\gamma (z)) \]
and $\ellpos (d_0)(\gamma (z)) = \gamma (d_0z)=0$ such that 
$[\gamma (a^\prime )] = [\gamma (a^{\prime \prime})]$.

We must also verify that the relations (\ref{rp1}) - (\ref{rp9}) are respected. For this we use that the
relations hold in $\ellpos (R_0)$. Let $b^\prime$ be a lift of $b$. For relation (\ref{rp1}) we have 
$\epsilon (a^\prime +b^\prime) = a+b$ such that
\[ r(\gamma (a+b)-\gamma (a)-\gamma (b)) = 
[\gamma (a^\prime +b^\prime )-\gamma (a^\prime )-\gamma (b^\prime ) ] = 0. \]
For (\ref{rp2}) we note that $\epsilon ((a^\prime )^2b^\prime ) = a^2 b$ such that
\[ r(\phi (a) \gamma (b)-\gamma (a^2b))= \phi (a)[\gamma (b^\prime )]-[\gamma ( (a^\prime )^2b^\prime )] = 
[\phi (a^\prime )\gamma (b^\prime )-\gamma ( (a^\prime )^2b^\prime )] = 0.\]
The arguments for the remaining relations are similar. 
\end{proof}

\begin{theorem}
There are natural isomorphisms of functors from unital (graded) commutative $\FF_2$-algebras of finite type to 
unital graded commutative $\FF_2$-algebras
\[ \ell \cong L_0 (HC_*^-), \quad \ellper \cong L_0 (HC_*^{per} ) \]
and a natural isomorphism from unital (graded) commutative $\FF_2$-algebras of finite type to graded 
$\FF_2$-vector spaces
\[ \ellpos \cong L_0(HC_*).\]
The last isomorphism is $\ell (A)$-linear when evaluated at any object $A$ from the domain category.
\end{theorem}

\begin{proof}
This follows from Proposition \ref{prop:0th}, Proposition \ref{prop:0th+} and Theorem \ref{thm:l-funktor_ss}.
\end{proof}

\begin{corollary}
\label{cor:universal}
The approximations $\psi : \ell \to HC_*^-$, $\psi^{per} : \ellper \to HC_*^{per}$ and 
$\psi^+ : \ellpos \to HC_*$ are universal in the following sense: Given functors 
$F$, $F^{per}$, $F^+$ between the respective categories together with natural transformations 
\[ f: F \to HC_*^- ,\quad  f^{per} : F^{per} \to HC_*^{per} , \quad f^+ : F^+ \to HC_* \] 
which are isomorphisms when evaluated on polynomial algebras. Then there are natural transformations $e: \ell \to F$, 
$e^{per}: \ellper \to F^{per}$ and $e^+ : \ellpos \to F^+$ such 
that the following diagrams commute:
\[ \xymatrix@C=0.8 cm{
\ell \ar[rd]^-{\psi} \ar[d]_-{e} & & & \ellper \ar[rd]^-{\psi^{per}} \ar[d]_-{e^{per}} & & & \ellpos \ar[rd]^-{\psi^{+}} \ar[d]_-{e^{+}} & \\
F \ar[r]_-{f} & HC_*^- & & F^{per} \ar[r]_-{f^{per}} & HC_*^{per} & & F^{+} \ar[r]_-{f^{+}} & HC_*^{+}
}\]
\end{corollary}

\begin{remark}
For a general commutative ring $k$ one can {\em define} approximation functors as follows:
\[ \ell = L_0 (HC_*^- ), \quad \ellper = L_0 (HC_*^{per} ), \quad \ellpos = L_0(HC_* ). \]
It could be interesting to have a presentations of these functors in terms of generators and relations as
we have when $k=\FF_2$. Such presentations could be viewed as Hochschild-Kostant-Rosenberg theorems
for the cyclic homology theories.
\end{remark}

\end{document}